\documentclass[sn-mathphys,Numbered]{sn-jnl} 

\usepackage{mathtools}
\usepackage{cite}
\usepackage{anyfontsize}
\usepackage{graphicx} 
\usepackage{multirow} 
\usepackage{amsmath,amssymb,amsfonts} 
\usepackage{amsthm} 
\usepackage{mathrsfs} 
\usepackage[title]{appendix} 
\usepackage{xcolor} 
\usepackage{float}
\usepackage{textcomp} 
\usepackage{manyfoot} 
\usepackage{booktabs} 
\usepackage{algorithm} 
\usepackage{algorithmicx} 
\usepackage{algpseudocode} 
\usepackage{listings}
\usepackage{array}
\usepackage{caption}
\usepackage{subcaption}
\usepackage[utf8]{inputenc}
\usepackage{booktabs}
\usepackage{amsfonts}
\usepackage{siunitx}
\usepackage{longtable}
\usepackage{eucal}
\usepackage{multirow}
\usepackage{enumerate}
\usepackage{lineno}
\usepackage{framed}
\usepackage{placeins}
\usepackage{cleveref} 
\newtheorem{theorem}{Theorem}[section]
\newtheorem{proposition}{Proposition}[section] 
\newtheorem{lemma}{Lemma}[section]

\newtheorem{remark}{Remark}[section] 
 
\newtheorem{definition}{Definition}[section]

\DeclareMathOperator{\argmin}{\textnormal{argmin}}
\begin{document}

\title[Article Title]{Newton Method for Multiobjective Optimization Problems of Interval-Valued Maps}

\author[1]{\fnm{Tapas} \sur{Mondal}}\email{tapas.ra.mat24@itbhu.ac.in}
\author[1]{\fnm{Debdas} \sur{Ghosh}}\email{debdas.mat@iitbhu.ac.in}
\author*[2]{\fnm{Do Sang} \sur{ Kim}}\email{dskim@pknu.ac.kr}

\affil[1]{\orgdiv{Department of Mathematical Sciences}, \orgname{Indian Institute of Technology (BHU)}, \orgaddress{ \city{Varanasi}, \postcode{221005}, \state{Uttar Pradesh}, \country{India}}}

\affil[2]{\orgdiv{Department of Applied Mathematics}, \orgname{Pukyong National University}, \orgaddress{ \city{Busan}, \postcode{48513}, \country{Korea}}}

\abstract{In this article, we propose a Newton-based method for solving {\it multiobjective interval optimization problems} (MIOPs). We first provide a connection between weakly Pareto optimal points and Pareto critical points in the context of MIOPs. Introducing this relationship, we develop an algorithm aimed at computing a Pareto critical point. The algorithm incorporates the computation of a descent direction at a non-Pareto critical point and employs an Armijo-like line search strategy to ensure sufficient decrease. Under suitable assumptions, we prove that the sequence generated by our proposed algorithm converges to a Pareto critical point. The effectiveness and performance of the proposed method are demonstrated through a series of numerical experiments on some test problems. Finally, we apply our proposed algorithm in a portfolio optimization problem with interval uncertainty.}



\maketitle
\section{Introduction}\label{Introduction}
In many real-world optimization scenarios, decision-making involves addressing multiple, often conflicting objectives rather than optimizing a single criterion. Such problems are known as multiobjective optimization problems. Unlike single-objective optimization, which typically yields a unique optimal solution, multiobjective optimization problems result in a set of trade-off solutions referred to as Pareto optimal solutions, where the improvement of one objective necessitates the deterioration of at least one other.

Moreover, uncertainty and imprecision—arising from limitations in measurement, estimation, or modeling—are inherent in various practical applications. These uncertainties can be effectively captured using interval-valued data, leading to the formulation of multiobjective optimization problems governed by {\it interval-valued mappings} (IVMs). In such problems, each objective function yields an interval rather than a precise scalar value, thereby requiring generalized notions of differentiability, optimality, and partial ordering for vectors with interval components.

The main objective of this article is to develop a Newton-type method for computing a Pareto critical point for MIOPs. In addition, we aim to establish the convergence of the sequence generated by the proposed algorithm to a Pareto critical point. Before proceeding, we provide a brief literature survey below.
\subsection{Literature Survey}
Over the years, numerous parameter-based and ordering-based approaches such as weighted sum, $\epsilon$-constraint, compromise programming, normal boundary intersection, normal constraint, physical programming, ideal cone, etc have been developed to determine the Pareto set in multiobjective optimization problems (for instance, see \cite{ehrgott2005multicriteria,miettinen1999nonlinear,ghosh2014directed}). However, these parameter-based and ordering-based methods have some disadvantages. The main disadvantage of the parameter-based method is that the choice of the parameters is not known in advance, and they depend on outside the problem data. On the other hand, giving the ordering importance to the objective functions is a burden task to the decision-maker. To overcome these drawbacks, some parameter-free as well as ordering-free optimization methods have been developed after the seminal work of Fliege and Svaiter \cite{fliege2000steepest}. These include Newton \cite{fliege2009newton}, quasi-Newton \cite{polvaj2014quasi}, conjugate gradient \cite{gonclaves2020extension}, projected gradient \cite{drummond2004projected,fukuda2013ineaxct}, trust region \cite{mohammadi2024trust}, conditional gradient \cite{assuncao2021conditional}, etc.  

 In 2009, Fliege et al. \cite{fliege2000steepest} first introduced the Newton method for multicriteria optimization problems to find the Pareto critical point. The authors solved a strongly convex subproblem to find the descent direction at a non-Pareto critical point and further used the Armijo-like rule to compute the step length. In addition, authors proved the convergence of the sequence generated by their proposed algorithm to the Pareto critical point. The authors proved the superlinear convergence result for twice continuously differentiable and locally strongly convex objective functions. In addition, the authors derived the quadratic rate of convergence with the assumption of Lipschitz continuity of the Hessian of objective functions. Further, Drummond et al. \cite{drummond2014quadratically} proposed a Newton-type method for smooth unconstrained vector optimization problems under general partial orders induced by closed convex cones, establishing global superlinear and local quadratic convergence under standard assumptions. In 2012, Wang and Liu \cite{wang2012regularized} proposed a regularized Newton method for convex multiobjective optimization problems, extending the Newton-type approach in \cite{fliege2009newton} by eliminating the requirement of strong convexity and ensuring convergence to weak Pareto-optimal points. In 2014, Fukuda and Drummond \cite{fukuda2014survey} analyzed in detail the steepest descent, the projected gradient, and the Newton methods through a rigorous and comprehensive survey. Lu and Chen \cite{lu2014newton} developed Newton-like methods for solving vector optimization problems, focusing on descent-type approaches that avoid scalarization and instead operate directly with vector-valued functions, particularly under convexity assumptions.
Wang et al. \cite{wang2019extended} developed extended Newton methods for multiobjective optimization problems by employing the majorizing function technique, establishing new semilocal, local, and global quadratic convergence results under the $L$-average Lipschitz condition, thereby significantly improving results in \cite{fliege2009newton}. However, the Majorization technique in \cite{wang2019extended} is applicable for convex function. Thus, in 2022, Gon\c{c}alves \cite{gonclaves2022globally} developed two Newton-based methods for multiobjective optimization problems that are globally convergent and applicable to nonconvex problems.  Recently, Ghosh \cite{ghosh2025cubic} proposed a cubic regularization of the Newton method to generate a set of weakly Pareto optimal points without convexity assumption on the objective function.

It is observed that traditional multiobjective optimization methods rely on deterministic data, assuming that all parameters are precisely known. However, real-world problems often involve uncertainties due to factors such as measurement errors, incomplete information, and system variability. Interval optimization provides an uncertain framework to model such uncertainties by representing uncertain parameters as intervals rather than exact values. To address uncertainties, interval optimization has gained attention in recent years. Studies by Moore \cite{moore1966interval} introduced foundational concepts in interval analysis, which have since been extended to multiobjective optimization problems. Nowadays, various techniques have been developed for solving interval optimization problems. Bhurjee and Panda \cite{bhurjee2012efficient} developed a parametric representation of interval-valued functions and proposed a methodology to solve general interval optimization problems, including convex quadratic programming, by converting them into equivalent deterministic optimization problems. Their approach ensures the existence and efficiency of solutions under appropriate convexity and order relations. Subsequently, Bhurjee and Panda \cite{bhurjee2016sufficient} introduced sufficient optimality conditions and duality theory for interval optimization problems. In 2017, Ghosh introduced the Newton \cite{ghosh2017newton} and quasi-Newton \cite{ghosh2017quasi} methods for capturing efficient solutions of an interval optimization problem. Further, Ghosh et al. \cite{ghosh2022generalized} developed a gradient-based algorithm to find the efficient solutions of an interval optimization problem. In 2024, Roy et al. \cite{roy2024gradient} proposed a gradient-based descent line search method under generalized Hukuhara differentiability for solving interval optimization problems and applied the method to finance.

 MIOPs are concerned with optimizing multiple objective functions simultaneously, where these functions are considered to be interval-valued rather than real-valued. MIOPs arise in various fields, including engineering, finance, logistics, and management, for instance, see \cite{upadhayay2024newton,giove2006interval,maity2019time}. To solve such MIOPs, Upadhyay et al. \cite{upadhayay2024newton} introduced Newton method and applied it to portfolio optimization. Subsequently, Upadhyay et al. \cite{upadhayay2024quasi} developed quasi-Newton method for MIOPs. Authors transformed an MIOP into a real-valued multiobjective optimization problem considering each objective function as a sum of lower and upper boundary functions of each IVM and further applied the Newton and quasi-Newton methods for transformed real-valued multiobjective optimization problem, which is a trivial case. Mondal and Ghosh \cite{mondal2025steepest} introduced the steepest descent method to find a Pareto critical point for an MIOP. The authors solved a strongly convex quadratic subproblem to compute a descent direction at a non-Pareto critical point. Further, introducing the steepest descent algorithm for an MIOP, they investigated the convergence properties of their proposed algorithm.

\subsection{Motivation and Work Done}

From the existing literature, we see that multiobjective optimization has been extensively studied in deterministic settings. However, these methods often assume precise parameter values, which may not hold in practical scenarios.  In the existing literature, there are steepest descent \cite{mondal2025steepest}, Newton \cite{upadhayay2024newton}, and quasi-Newton \cite{upadhayay2024quasi} methods for MIOPs. However, in \cite{upadhayay2024newton,upadhayay2024quasi}, authors transformed an MIOP into a real-valued multiobjective optimization problem considering each objective function as a sum of the lower and upper boundary functions of each IVM and further applied the Newton and quasi-Newton methods for real-valued multiobjective optimization, which essentially captures very tiny part of the entire set of efficient solutions. One can trivially find that almost entire part of the efficient solutions cannot be captured by the methods in \cite{upadhayay2024newton,upadhayay2024quasi}. Although in steepest descent method \cite{mondal2025steepest} for an MIOP is able to capture the almost every Pareto optimal points, it used only the generalized Hukuhara gradient information. Thus, we make an attempt to study the Newton method for MIOPs to capture almost entire part of the efficient solutions using generalized Hukuhara gradient and generalized Hukuhara Hessian information. In this study, our main contributions are as follows.
\begin{enumerate}
	\item[(i)] We propose an algorithm for the Newton method to find a Pareto critical point of an MIOP. In the proposed algorithm, we give the computation of a descent direction at a non-Pareto critical point. We prove the result related to the descent direction calculation. Further, we use the Armijo-like rule to find the step length. In addition, we prove the existence of the step length.
	\item[(ii)] We show that the iteration scheme of our proposed algorithm is scaling independent of variable.
	\item[(iii)] We prove that under certain reasonable assumption, the sequence generated by our proposed algorithm converges to a Pareto critical point.
	\item[(iv)] We provide computational experiments through some test problems to validate our proposed algorithm.
	\item[(v)] We give a portfolio optimization problem in a an MIOP framework and apply our proposed algorithm. 
\end{enumerate}
\subsection{Delineation}
The remainder of the paper is structured as follows. In Section \ref{Preliminaries}, preliminaries on interval analysis and basic definitions and results on MIOP are given. In Section \ref{Newton method}, the Newton method for an MIOP is developed. In Section \ref{Convergence analysis}, the convergence analysis is provided. In Section \ref{Numerical experiments}, the numerical performance of the proposed algorithm is shown through some test problems. In addition, we give an application of our proposed method in a portfolio optimization problem in Section \ref{Application}. Finally, the study is concluded in Section \ref{Conclusion and future directions}, where some further research scopes are also given.

\section{Preliminaries}\label{Preliminaries}
In this section, we give basic ideas of interval analysis. In addition, we give some definitions and basic results on MIOP. Throughout, we use the notations ${\mathbb{R}}$, ${\mathbb{R}_+}$, and ${\mathbb{N}}$ to mean the set of real numbers, the set of nonnegative real numbers, and the set of natural numbers, respectively. In addition, we denote the zero interval by ${\bf 0}$, i.e., ${\bf 0}=[0,0]$.

\subsection{Interval Analysis}\label{interval analysis}
Let ${\it I}({\mathbb{R}})$ be the collection of all closed and bounded intervals. Consider a real number $\lambda$ and a pair of intervals $ S:=\left[\underline{s},\overline{s}\right]$ and $ T:=\left[\underline{t},\overline{t}\right]$. Then, as per Moore {\normalfont\cite{moore1966interval}}, interval addition, subtraction, multiplication, and scalar multiplication are written by $ S\oplus T$, $ S\ominus T$, $ S\odot T$, and $\lambda\odot S$, respectively, which are defined as follows:
\begin{enumerate}
	\item[(i)] $ S\oplus T:=\left[\underline{s}+\underline{t},\overline{s}+\overline{t}\right]$;
	\item[(ii)] $ S\ominus T:=\left[\underline{s}-\overline{t},\overline{s}-\underline{t}\right]$;
	\item[(iii)] $ S\odot T:=\left[\min\:\left\{\underline{s}\:\underline{t},\underline{s}\:\overline{t},\overline{s}\:\underline{t},\overline{s}\:\overline{t}\right\},\max\:\left\{\underline{s}\:\underline{t},\underline{s}\:\overline{t},\overline{s}\:\underline{t},\overline{s}\:\overline{t}\right\}\right]$;
	\item[(iv)] 
	$\lambda\odot S: = \begin{cases}
		\left[\lambda\underline{s},\lambda\overline{s}\right]  & \text{if } \lambda \geq 0, \\
		\left[\lambda\overline{s},\lambda\underline{s}\right]  & \text{if } \lambda<0. 
	\end{cases} $
\end{enumerate}

\medskip
\begin{definition}[$gH$-difference {\normalfont\cite{stefanini2008generalization}}]\label{gH difference definition}
	\normalfont
	Let $ S$, $ T$, and $ V$ be three elements of ${\it I}({\mathbb{R}})$. If $ S:= T\oplus  V$ or $ T:=S\ominus  V$, then $ V$ is called the $gH$-difference between $ S$ and $ T$. We write it as $ V:= S\ominus_{gH} T$. For a given pair of intervals $ S:=\left[\underline{s},\overline{s}\right]$ and $ T:=\left[\underline{t},\overline{t}\right]$, the $gH$-difference is calculated by \[ S\ominus_{gH}T:=\left[\min\:\left\{\underline{s}-\underline{t},\overline{s}-\overline{t}\right\},\max\:\left\{\underline{s}-\underline{t},\overline{s}-\overline{t}\right\}\right].\] 
\end{definition}

\medskip
The following presents a dominance relation for intervals in the context of a minimization problem---smaller is better. 
\medskip

\begin{definition}[Dominance relation of intervals {\normalfont\cite{chauhan2021generalized}}]\label{Dominance definition}
	\normalfont
	Consider a pair of elements $ S:=\left[\underline{s},\overline{s}\right]$ and $ T:=\left[\underline{t},\overline{t}\right]$ from $ {\it I}({\mathbb{R}})$.
	\begin{enumerate}
		\item[(i)]  If $\underline{s}\geq\underline{t}$ and $\overline{s}\geq\overline{t}$, then we say that $ T$ dominates $ S$, and represent it by $ S\succeq  T$.
		\item[(ii)] If either $\left(\underline{s}>\underline{t} \text{ and }\overline{s}\geq\overline{t}\right)$ or $\left(\underline{s}\geq\underline{t} \text{ and }\overline{s}>\overline{t}\right)$, then we say that $ T$ strictly dominates $ S$, and express it by $ S\succ  T$.
		\item[(iii)] If $ T$ does not dominate $ S$, then we present it by $ S\nsucceq  T$.
		\item[(iv)] If $ T$ does not strictly dominate $ S$, then we denote it by $S\nsucc  T$. 
		\item[(v)] If $ S$ dominates $T$ or $ T$ dominates $ S$, then we say that $ S$ and $ T$ are comparable.
		\item[(vi)] If $ S$ does not dominate $ T$ and $ T$ does not dominate $ S$, then we say that $ S$ and $T$ are not comparable.
	\end{enumerate}
\end{definition}

\medskip
The dominance $ S\succeq  T$ is also interchangeably presented by $ T\preceq S$. Similarly, $S\succ T$, $ S\nsucceq  T$, and $ S\nsucc  T$ are alternatively presented by $ T\prec S$, $ T\npreceq  S$, and $ T\nprec  S$, respectively.

\medskip
\begin{definition}[Norm on ${{\it I}({\mathbb{R}})}$ {\normalfont\cite{moore1966interval}}]\label{Norm on IR definition}
	\normalfont
	The norm on ${{\it I}({\mathbb{R}})}$ is a function $\|\cdot\|_{{{\it I}({\mathbb{R}})}}:{{\it I}({\mathbb{R}})}\to{\mathbb{R}}_+$ given by \[\|S\|_{{\it I}({\mathbb{R}})}:=\max\:\left\{|\underline{s}|,|\overline{s}|\right\},S:=\left[\underline{s},\overline{s}\right]\in {{\it I}({\mathbb{R}})}.\]
\end{definition}

\medskip
\begin{definition}[Norm on ${{\it I}({\mathbb{R}})^n}$ {\normalfont\cite{ghosh2022generalized}}]\label{Norm on IRn definition}
	\normalfont
	The norm on ${{\it I}({\mathbb{R}})^n}$ is a function $\|\cdot\|_{{{\it I}({\mathbb{R}})^n}}:{{\it I}({\mathbb{R}})^n}\to{\mathbb{R}}_+$ given by \[\| {\tilde S}\|_{{\it I}({\mathbb{R}})^n}:=\left\| S_1\right\|_{{\it I}\left({\mathbb{R}}\right)}+\left\| S_2\right\|_{{\it I}\left({\mathbb{R}}\right)}+\cdots+\left\| S_n\right\|_{{\it I}\left({\mathbb{R}}\right)},   {\tilde{S}}:=\left( S_1, S_2,\ldots,  S_n\right)^\top\in {\it I}({\mathbb{R}})^n.\] 
\end{definition}

\medskip
Throughout the article, the notation $\|\cdot\|$ represents the usual Euclidean norm in ${\mathbb{R}}^n$. Let $ H:{\mathbb{R}}^n\to {\it I}({\mathbb{R}})$ be an IVM, which is represented by $H:=\left[\underline{H},\overline{H}\right]$, where $\underline{H}:{\mathbb{R}}^n\to{\mathbb{R}}$ and $\overline{H}:{\mathbb{R}}^n\to{\mathbb{R}}$ are two real-valued functions with $\underline{H}\left(x\right)\leq \overline{H}\left(x\right)$ for all $x\in{\mathbb{R}}^n$. The functions $\underline{H}$ and $\overline{H}$ are called lower and upper boundary functions of the IVM $H$, respectively.

\medskip
\begin{definition}[$gH$-continuity {\normalfont\cite{ghosh2017newton}}]\label{gH continuity of IVM definition}
	\normalfont
		An IVM $H:{\mathbb{R}}^n\rightarrow {\it I}({\mathbb{R}})$ is said to be $gH$-continuous at $ \tilde x$ if \[\underset{\| p\|\to 0}{\lim}\left( H( \tilde x+  p)\ominus_{gH} H( \tilde x)\right)={\bf 0}.\]
\end{definition}

\medskip
	\begin{definition}[$gH$-Lipschitz continuity {\normalfont\cite{ghosh2022generalized}}]\label{gH Lipschitz continuity of IVM definition}
	\normalfont
	An IVM $H:{\mathbb{R}}^n\rightarrow {\it I}({\mathbb{R}})$ is said to be $gH$-Lipschitz continuous if there exists an $L>0$ such that \[\|H(x)\ominus_{gH}H(y)\|_{{{\it I}({\mathbb{R}})}}\leq L\|x-y\| \text{ for all }x,y \in{\mathbb{R}}^n.\]
\end{definition}

\medskip
The interrelation between $gH$-continuity ($gH$-Lipschitz continuity) of an IVM and continuity (Lipschitz continuity) of its boundary functions are given in the following lemma.

\medskip
\begin{lemma}\label{convexity, continuity with boundary functions}
	Let $H:{\mathbb{R}}^n\rightarrow {\it I}({\mathbb{R}})$ be an IVM given by $H:=[\underline{H},\overline{H}]$. Then, the following results hold. 
	\begin{enumerate}
		\item[(i)] Both of boundary functions $\underline{H}$ and $\overline{H}$ are continuous at ${ \bar x}$ if and only if the IVM $H$ is $gH$-continuous at ${ \bar x}$ (see {\normalfont\cite{ghosh2017newton}}).
		\item[(ii)] Both of boundary functions $\underline{H}$ and $\overline{H}$ are Lipschitz continuous if and only if the IVM $H$ is $gH$-Lipschitz continuous (see {\normalfont\cite{ghosh2022generalized}}).
	\end{enumerate}
\end{lemma}

\medskip
The following are fundamental concepts and important to proceed further.
\medskip
\begin{definition}[$gH$-derivative {\normalfont\cite{debnath2022generalized}}]\label{gH derivative of IVM definition}
	\normalfont
	Let $\mathcal{X}\subseteq \mathbb{R}$ be an open set and $H:\mathcal{X}\to{\it I}\left(\mathbb{R}\right)$ be an IVM. If $H'(\tilde{x}):=\underset{ p\to 0}{\lim}\tfrac{1}{p}\odot\left( H( \tilde x+  p)\ominus_{gH} H( \tilde x)\right)$ exists, then it is called $gH$-derivative of $H$ at $\tilde{x}\in\mathcal{X}$. 
\end{definition}

\medskip
\begin{remark}
	\normalfont
	Let $\mathcal{X}\subseteq \mathbb{R}$ be an open set and $H:\mathcal{X}\to{\it I}\left(\mathbb{R}\right)$ be an IVM given by $H:=\left[\underline{H},\overline{H}\right]$. If the derivatives of $\underline{H}\text{ and }\overline{H}$ at $\tilde{x}\in\mathcal{X}$ exist, then $gH$-derivative of $H$ exists. Moreover, if the derivatives of $\underline{H}\text{ and }\overline{H}$ are $\underline{H}'\text{ and }\overline{H}'$, respectively, then \[H'(\tilde{x}):=\left[\min\left\{\underline{H}'(\tilde{x}),\overline{H}'(\tilde{x})\right\},\max\left\{\underline{H}'(\tilde{x}),\overline{H}'(\tilde{x})\right\}\right].\]
\end{remark}

\medskip
\begin{definition}[$gH$-partial derivative {\normalfont\cite{debnath2022generalized}}]\label{gH partial derivative of IVM definition}
	\normalfont
	Let $\mathcal{X}\subseteq {\mathbb{R}}^n$ be an open set and $H:\mathcal{X}\to{\it I}\left(\mathbb{R}\right)$ be an IVM. We define a function $H_i(x_i):=H(\tilde{x}_1,\tilde{x}_2,\ldots,\tilde{x}_{i-1},x_i,\tilde{x}_{i+1},\ldots,\tilde{x}_n)$, where $\tilde{x}:=(\tilde{x}_1,\tilde{x}_2,\ldots,\tilde{x}_n)^\top$. If $H_i'$ exists at $\tilde{x}_i$, then $i$-th $gH$-partial derivative of $H$ at $\tilde{x}$, denoted by $D_iH(\tilde{x})$, is defined by \[D_iH(\tilde{x}):=H_i'(\tilde{x}_i)\text{ for all }i=1,2,\ldots,n.\]
\end{definition}

\medskip
\begin{definition}[$gH$-gradient {\normalfont\cite{debnath2022generalized}}]\label{gH gradient of IVM definition}
	\normalfont
	Let $\mathcal{X}\subseteq {\mathbb{R}}^n$ be an open set and $H:\mathcal{X}\to{\it I}\left(\mathbb{R}\right)$ be an IVM. The $gH$-gradient of $H$ at $\tilde{x}\in\mathcal{X}$, denoted by $\nabla_{gH}H(\tilde{x})$, is defined by \[\nabla_{gH}H(\tilde{x}):=\left(D_1H(\tilde{x}),D_2H(\tilde{x}),\ldots,D_nH(\tilde{x})\right)^\top.\] 
\end{definition}

\medskip
\begin{definition}[$gH$-Lipschitz gradient {\normalfont\cite{ghosh2022generalized}}]\label{gH Lipschitz gradient of IVM definition}
	\normalfont
	Let $\mathcal{X}\subseteq {\mathbb{R}}^n$ be an open set. An IVM $H:\mathcal{X}\to{\it I}\left(\mathbb{R}\right)$ is said to have $gH$-Lipschitz gradient on $\mathcal{X}$ if there exists an $L>0$ such that \[\left\|\nabla_{gH}H(x)\ominus_{gH}\nabla_{gH}H(y)\right\|_{\it I\left({\mathbb{R}}\right)^n}\leq L\left\|x-y\right\| \text{ for all }x,y\in\mathcal{X}.\]
\end{definition}

\medskip
\begin{definition}[$gH$-Hessian {\normalfont\cite{upadhayay2024quasi}}]\label{gH Hessian of IVM definition}
	\normalfont
	Let $\mathcal{X}\subseteq {\mathbb{R}}^n$ be an open set and $H:\mathcal{X}\to{\it I}\left(\mathbb{R}\right)$ be an IVM. If all the second order $gH$-partial derivatives of $H$ at $\tilde{x}\in\mathcal{X}$ exist, then the $gH$-Hessian of $H$ at $\tilde{x}\in\mathcal{X}$, denoted by $\nabla_{gH}^2H(\tilde{x})$, is defined by 
	\[\nabla_{gH}^2H(\tilde{x}):=\left[\underline{\nabla_{gH}^2H(\tilde{x})},\overline{\nabla_{gH}^2H(\tilde{x})}\right]=\left[\min\left\{\frac{\partial^2 \underline{H}(\tilde{x})}{\partial x_{r}\partial x_{s}},\frac{\partial^2 \overline{H}(\tilde{x})}{\partial x_r \partial x_s}\right\},\max\left\{\frac{\partial^2 \underline{H}(\tilde{x})}{\partial x_{r}\partial x_{s}},\frac{\partial^2 \overline{H}(\tilde{x})}{\partial x_r \partial x_s}\right\}\right]_{1\leq r,s\leq n}.\]
\end{definition}

\medskip
\begin{definition}[Linear IVM {\normalfont\cite{ghosh2022generalized}}]\label{Linear IVM definition}
	\normalfont
	Let $\mathcal{X}$ be a linear subspace of ${\mathbb{R}}^n$. An IVM $H:\mathcal{X}\to{\it I}\left(\mathbb{R}\right)$ is said to be linear if \[H(x):=\bigoplus_{j=1}^{n}H(e_j)\odot x_j \text{ for all } x:=(x_1,x_2,\ldots,x_n)^\top \in \mathcal{X},\]
	where $e_j$ is the $j$-th standard basis vector of ${\mathbb{R}}^n$, $j=1,2,\ldots,n$ and `$\bigoplus_{j=1}^{n}$' denotes successive addition of $n$ number of intervals.
\end{definition}
\medskip
\begin{definition}[$gH$-differentiable IVM {\normalfont\cite{ghosh2022generalized}}]\label{$gH$-differentiabe IVM definition}
	\normalfont
	Let $\mathcal{X}\subseteq {\mathbb{R}}^n$ be an open set. An IVM $H:\mathcal{X}\to{\it I}\left(\mathbb{R}\right)$ is said to be $gH$-differentiable at a point $\tilde{x}\in\mathcal{X}$ if there exists a linear IVM $T_{\tilde{x}}:{\mathbb{R}}^n\to {\it I}\left({\mathbb{R}}\right)$, an IVM $E\left(H\left(\tilde{x};v\right)\right)$, and a $\delta>0$ such that  \[H\left(\tilde{x}+v\right)\ominus_{gH}H\left(\tilde{x}\right):=T_{\tilde{x}}\left(v\right)\oplus\left\|v\right\|\odot E\left(H\left(\tilde{x};v\right)\right)\text{ for all } v \text{ with } \left\|v\right\|<\delta,\]
	where $E\left(H\left(\tilde{x};v\right)\right)\to{\bf 0}$ as $\left\|v\right\|\to 0.$
	
	\bigskip
	
	\noindent
	If the IVM $H:\mathcal{X}\to{\it I}\left(\mathbb{R}\right)$ is $gH$-differentiable at each point $\tilde{x}\in\mathcal{X}$, then the IVM $H$ is called $gH$-differentiable on $\mathcal{X}$.
\end{definition}
\medskip
\begin{lemma}[ {\normalfont\cite{ghosh2022generalized}}]\label{linear IVM lemma}
	Let $\mathcal{X}\subseteq {\mathbb{R}}^n$ be an open set. If the IVM $H:\mathcal{X}\to{\it I}\left(\mathbb{R}\right)$ is $gH$-differentiable at a point $\tilde{x}\in\mathcal{X}$, then there exists an $\alpha$ and a $\delta>0$ such that 
	\[\underset{\alpha\to0}{\lim}\tfrac{1}{\alpha}\odot\left(H\left(\tilde{x}+\alpha v\right)\ominus_{gH}H\left(\tilde{x}\right)\right):=T_{\tilde{x}}\left(v\right) \text{ for all } v\in {\mathbb{R}}^n \text{ with } \left|\alpha\right|\left\|v\right\|<\delta,\]
	where $T_{\tilde{x}}\left(v\right)$ is the linear IVM.
	\bigskip
	
	\noindent Moreover, if the $gH$-gradient of $H$ exists at the point $\tilde{x}\in\mathcal{X}$, then the linear IVM $T_{\tilde{x}}\left(v\right)$ is expressed by 
	\[T_{\tilde{x}}\left(v\right):=\nabla_{gH} H\left(\tilde{x}\right)^\top\odot v \text{ for all } v:=\left(v_1,v_2,\ldots,v_n\right)^\top \in {\mathbb{R}}^n,\]
	where $\nabla_{gH} H\left(\tilde{x}\right)^\top\odot v:=\bigoplus_{j=1}^n  D_jH\left(\tilde{x}\right)\odot v_j.$
\end{lemma}

\medskip
\begin{definition}[Convex IVM {\normalfont\cite{wu2007karush}}]\label{Convex IVM definition}
	\normalfont
	Let $\mathcal{X}\subseteq {\mathbb{R}}^n$ be convex. An IVM $H:\mathcal{X}\to{\it I}\left(\mathbb{R}\right)$ is said to be convex if \[ H\left(\theta x+\left(1-\theta\right)y\right)\preceq\theta\odot H( x)\oplus(1-\theta)\odot H(y) \text{ for all } x,y\in \mathcal{X} \text{ and } \theta \in [0,1].\]
\end{definition}
\medskip
\begin{lemma}[ {\normalfont\cite{ghosh2022generalized}}]\label{convex differentiable lemma}
	Let $\mathcal{X}\subseteq {\mathbb{R}}^n$ be an open and convex set. If the IVM $H:\mathcal{X}\to{\it I}\left(\mathbb{R}\right)$ is convex and $gH$-differentiable on $\mathcal{X}$, then 
	\[H\left(y\right)\succeq H\left(x\right)\oplus\nabla_{gH}H\left(x\right)^\top\odot\left(y-x\right)\text{ for all } x,y\in{\mathcal{X}}.\]
\end{lemma}
\medskip
\begin{definition}[Strongly convex IVM {\normalfont\cite{upadhayay2024quasi}}]\label{Strongly convex IVM definition}
	\normalfont
	Let $\mathcal{X}\subseteq {\mathbb{R}}^n$ be convex. An IVM $H:\mathcal{X}\to{\it I}\left(\mathbb{R}\right)$ is said to be strongly convex with modulus $\gamma>0$ if $H(x)\ominus_{gH} \left[\tfrac{\gamma}{2},\tfrac{\gamma}{2}\right]\odot \|x\|^2$ is convex in $\mathcal{X}$.
\end{definition}

\medskip
\subsection{Multiobjective Interval Optimization Problem}
Let $U\subseteq{\mathbb{R}}^n$ be an open set and $G:U\rightarrow{\it I}\left({\mathbb{R}}\right)^m$ be a multiobjective IVM given by $G:=\left(G_1,G_2,\ldots,G_m\right)^\top$, where $G_i:U\rightarrow{\it I}\left({\mathbb{R}}\right)$ is a twice $gH$-continuously differentiable and locally strongly convex IVM given by $G_i:=\left[\underline{G}_i,\overline{G}_i\right]$ for all $i=1,2,\ldots,m$.
 We consider to solve the following MIOP in this study:
\begin{align}\label{minG(x)}
	\underset{x\in U}{\min}\: G(x).
\end{align}
Let us now define the solution concepts -- weakly Pareto optimal, Pareto optimal, and Pareto critical points for the MIOP \eqref{minG(x)}.

\medskip
\begin{definition}[Weakly Pareto optimal point{\normalfont \cite{mondal2025steepest}}]\label{weakly Pareto optimal}
	\normalfont
	An $x^\star\in U$ is said to be a weakly Pareto optimal point of the MIOP \eqref{minG(x)} if there does not exist any other $x\in U$ such that $G_i(x)\prec G_i(x^\star)$ for all $i=1,2,\ldots,m$.
\end{definition} 

\medskip
\begin{definition}[Pareto optimal point {\normalfont \cite{mondal2025steepest}}]\label{Pareto optimal}
	\normalfont
	An $x^\star\in U$ is said to be a Pareto optimal point of the MIOP \eqref{minG(x)} if there does not exist any other $x\in U$ such that $G_i(x)\preceq G_i(x^\star)$ for all $i=1,2,\ldots,m$.
\end{definition} 
\medskip
\begin{remark}
	\normalfont
	An $x^\star\in U$ is said to be a locally Pareto optimal point (respectively, locally weakly Pareto optimal point) of the MIOP \eqref{minG(x)} if there exists a neighborhood $V\subseteq U$ of $x^\star$ such that the point $x^\star$ is a Pareto optimal (respectively, weakly Pareto optimal ) of the MIOP \eqref{minG(x)} restricted to $V$. Note that if $U$ is convex and $G$ is $I\left({\mathbb{R}}\right)^m$ convex, i.e., $G_1,G_2,\ldots,G_m$ are all convex IVM, then each local Pareto optimal point is globally Pareto optimal.
\end{remark}
\medskip
\begin{definition}[Pareto critical point {\normalfont \cite{mondal2025steepest}}]\label{Pareto critical}
	\normalfont
	A point $x^\star\in U$ is said to be Pareto critical point of the MIOP \eqref{minG(x)} if there does not exist any $v\in{\mathbb{R}}^n$ such that $\nabla_{gH} G_i(x^\star)^\top\odot v\prec{\bf 0}$ for all $i=1,2,\ldots,m$.
\end{definition}
\medskip
\begin{definition}[Descent direction {\normalfont \cite{mondal2025steepest}}]\label{Descent direction}
	\normalfont
	A direction vector $v\in{\mathbb{R}}^n$ is said to be descent direction of the MIOP \eqref{minG(x)} at a point $x\in{\mathbb{R}}^n$ if there exists a $\delta>0$ such that \[G_i\left(x+tv\right)\prec G_i\left(x\right) \text{ for all } t\in\left(0,\delta\right) \text{and for all }i=1,2,\ldots,m.\] 
\end{definition}

\medskip
Before proceeding the following result, we use the notation $G\in C_{gH}^k\left(U,I\left({\mathbb{R}}\right)^m\right)$ to mean $G_1,G_2,\ldots,G_m$ are all $k$-times $gH$-continuously differentiable IVMs. Unless explicitly mentioned, we assume that $\nabla_{gH}^2G_i(x)$ is positive definite for all $x\in U$ and for all $i=1,2,\ldots,m$, i.e., \[v^\top \odot \nabla_{gH}^2G_i(x)\odot v\succ {\bf 0}\text{ for all nonzero } v\in {\mathbb{R}}^n,\text{ for all } x\in U,\text{ and for all }i=1,2,\ldots,m. \]
Under this assumption, $G$ is ${\it I}\left({\mathbb{R}}\right)^m$ convex on each convex subset of $U$. Throughout the article, we use the notation $\nabla_{gH}^2G_i(x)\succ{\bf 0}$ to mean that  $\nabla_{gH}^2G_i(x)$ is positive definite.
\medskip
\begin{lemma}[{\normalfont \cite{mondal2025steepest}}]\label{interrelation lemma of Pareto optimal and critical}
	Assume that $G\in C_{gH}^1\left(U,I\left({\mathbb{R}}\right)^m\right).$
	\begin{itemize}
		\item[(i)]  If $x^\star$ is a locally weak Pareto optimal point of the MIOP \eqref{minG(x)}, then it is a Pareto critical point of the MIOP \eqref{minG(x)}.
		\item[(ii)]  If $U$ is convex, $G$ is $I\left({\mathbb{R}}\right)^m$ convex, and $x^\star\in U$ is a Pareto critical point of the MIOP \eqref{minG(x)}, then $x^\star$ is a weakly Pareto optimal point of the MIOP \eqref{minG(x)}.
	\end{itemize}
\end{lemma}
\medskip
\begin{proposition}\label{Pareto critical to Pareto optimal}
	If $U$ is convex, $G\in C_{gH}^2\left(U,I\left({\mathbb{R}}\right)^m\right)$, $\nabla_{gH}^2G_i\left(x\right)\succ {\bf 0}$ for all $i=1,2,\ldots,m$ and all $x\in U$, and if $x^\star\in U$ is a Pareto critical point of the MIOP \eqref{minG(x)}, then $x^\star$ is a Pareto optimal point of the MIOP \eqref{minG(x)}.
\end{proposition}
\medskip
\begin{proof}
The proof is similar to the proof of item (ii) of Lemma \ref{interrelation lemma of Pareto optimal and critical}.
\end{proof}

\section{Newton Method}\label{Newton method}
In this section, we develop the Newton method for capturing the Pareto critical points of the MIOP \eqref{minG(x)}. We first focus on computing a descent direction at a non-Pareto critical point, and then show the computation of the step length along with the step-wise algorithm.

\subsection{Computing Descent Direction}
We now proceed to compute the Newton direction for the MIOP \eqref{minG(x)}. Before move forward to identify such direction $v\in{\mathbb{R}}^n$ at a given point $x\in U$, we define an IVM $g_x^i:{\mathbb{R}}^n\rightarrow {\it I}({\mathbb{R}})$  for each $i=1,2,\ldots,m$ by 
 \begin{align}\label{gxi}
 g_x^i(v):=\nabla_{gH} G_i(x)^\top\odot v\oplus \tfrac{1}{2}\odot v^\top \odot \nabla_{gH}^2 G_i(x)\odot v .
\end{align}
Denote the $r$-th component of $\nabla_{gH}G_i(x)$ by $$\left[\underline{\nabla_{gH}G_i(x)_r},\overline{\nabla_{gH}G_i(x)_r}\right]:=\left[\min\left\{\frac{\partial \underline{G}_i(x)}{\partial x_r},\frac{\partial \overline{G}_i(x)}{\partial x_r}\right\},\max\left\{\frac{\partial \underline{G}_i(x)}{\partial x_r},\frac{\partial \overline{G}_i(x)}{\partial x_r}\right\}\right]$$ and the $(r,s)$ element of $\nabla_{gH}^2G_i(x)$ by $$\left[\underline{\nabla_{gH}^2G_i(x)_{rs}},\overline{\nabla_{gH}^2G_i(x)_{rs}}\right]:=\left[\min\left\{\frac{\partial^2 \underline{G}_i(x)}{\partial x_{r}\partial x_{s}},\frac{\partial^2 \overline{G}_i(x)}{\partial x_r \partial x_s}\right\},\max\left\{\frac{\partial^2 \underline{G}_i(x)}{\partial x_{r}\partial x_{s}},\frac{\partial^2 \overline{G}_i(x)}{\partial x_r \partial x_s}\right\}\right].$$ Let $\underline{g}_x^i$ and $\overline{g}_x^i$ be the lower and the upper boundary functions of the IVM $g_x^i:{\mathbb{R}}^n\to {\it I}\left({\mathbb{R}}\right)$. Then, we have 
\begin{equation}\label{g_xilower}
	\begin{aligned}
	\underline{g}_x^i(v):=&\sum_{r=1}^{n}\min\left\{\underline{\nabla_{gH} G_i(x)_r}v_r,\overline{\nabla_{gH} G_i(x)_r}v_r\right\}\\
	&+\tfrac{1}{2}\sum_{r=1}^{n}\sum_{s=1}^{n}\min\left\{\underline{\nabla_{gH}^2 G_i(x)_{rs}}v_r v_s,\overline{\nabla_{gH}^2 G_i(x)_{rs}}v_r v_s\right\}
\end{aligned}
	\end{equation}
	 and 
	\begin{equation}\label{g_xiupper}
		\begin{aligned}
		\overline{g}_x^i(v):=&\sum_{r=1}^{n}\max\left\{\underline{\nabla_{gH} G_i(x)_r}v_r,\overline{\nabla_{gH} G_i(x)_r}v_r\right\}\\
	 &+\tfrac{1}{2}\sum_{r=1}^{n}\sum_{s=1}^{n}\max\left\{\underline{\nabla_{gH}^2 G_i(x)_{rs}}v_r v_s,\overline{\nabla_{gH}^2 G_i(x)_{rs}}v_r v_s\right\}.
	\end{aligned}
\end{equation}
Denote $\left|v\right|:=\left(\left|v_1\right|,\left|v_2\right|,\ldots,\left|v_n\right|\right)^\top$. Then, we have
\begin{align*}
	\underline{g}_x^i(v)~=~&\tfrac{1}{2}\sum_{r=1}^{n}\left[\left(\underline{\nabla_{gH} G_i(x)_r}v_r+\overline{\nabla_{gH} G_i(x)_r}v_r\right)-\left|{\overline{\nabla_{gH} G_i(x)_r}v_r-\underline{\nabla_{gH} G_i(x)_r}v_r}\right|\right]\\
	&+\tfrac{1}{4}\sum_{r=1}^{n}\sum_{s=1}^{n}\left[\left(\underline{\nabla_{gH}^2 G_i(x)_{rs}}v_rv_s+\overline{\nabla_{gH}^2 G_i(x)_{rs}}v_r v_s\right)-\left|{\overline{\nabla_{gH}^2 G_i(x)_{rs}}v_rv_s-\underline{\nabla_{gH}^2 G_i(x)_{rs}}v_rv_s}\right|\right]   \\
	~=~&  \tfrac{1}{2}\left(\underline{\nabla_{gH} G_i(x)}+\overline{\nabla_{gH} G_i(x)}\right)^\top v-\tfrac{1}{2}\left({\overline{\nabla_{gH} G_i(x)}-\underline{\nabla_{gH} G_i(x)}}\right)^\top \left|v\right|\\
	&+  \tfrac{1}{4} v^\top\left(\underline{\nabla_{gH}^2 G_i(x)}+\overline{\nabla_{gH}^2 G_i(x)}\right) v-\tfrac{1}{4}\left|v\right|^\top\left({\overline{\nabla_{gH}^2 G_i(x)}-\underline{\nabla_{gH}^2 G_i(x)}}\right) \left|v\right|           
\end{align*}
and 
\begin{align*}
	\overline{g}_x^i(v)~=~&\tfrac{1}{2}\sum_{r=1}^{n}\left[\left(\underline{\nabla_{gH} G_i(x)_r}v_r+\overline{\nabla_{gH} G_i(x)_r}v_r\right)+\left|{\overline{\nabla_{gH} G_i(x)_r}v_r-\underline{\nabla_{gH} G_i(x)_r}v_r}\right|\right]\\
	&+\tfrac{1}{4}\sum_{r=1}^{n}\sum_{s=1}^{n}\left[\left(\underline{\nabla_{gH}^2 G_i(x)_{rs}}v_rv_s+\overline{\nabla_{gH}^2 G_i(x)_{rs}}v_r v_s\right)+\left|{\overline{\nabla_{gH}^2 G_i(x)_{rs}}v_rv_s-\underline{\nabla_{gH}^2 G_i(x)_{rs}}v_rv_s}\right|\right]   \\
	~=~&  \tfrac{1}{2}\left(\underline{\nabla_{gH} G_i(x)}+\overline{\nabla_{gH} G_i(x)}\right)^\top v+\tfrac{1}{2}\left({\overline{\nabla_{gH} G_i(x)}-\underline{\nabla_{gH} G_i(x)}}\right)^\top \left|v\right|\\
	&+  \tfrac{1}{4} v^\top\left(\underline{\nabla_{gH}^2 G_i(x)}+\overline{\nabla_{gH}^2 G_i(x)}\right) v+\tfrac{1}{4}\left|v\right|^\top\left({\overline{\nabla_{gH}^2 G_i(x)}-\underline{\nabla_{gH}^2 G_i(x)}}\right) \left|v\right|.
\end{align*}
 Therefore, we get
 \begin{equation}\label{gxi-lower-upper}
 	\begin{rcases}
 		\begin{aligned}
 			\underline{g}_x^i(v):=&\tfrac{1}{2}\left(\underline{\nabla_{gH} G_i(x)}+\overline{\nabla_{gH} G_i(x)}\right)^\top v-\tfrac{1}{2}\left({\overline{\nabla_{gH} G_i(x)}-\underline{\nabla_{gH} G_i(x)}}\right)^\top \left|v\right|\\
 			&+  \tfrac{1}{4} v^\top\left(\underline{\nabla_{gH}^2 G_i(x)}+\overline{\nabla_{gH}^2 G_i(x)}\right) v-\tfrac{1}{4}\left|v\right|^\top\left({\overline{\nabla_{gH}^2 G_i(x)}-\underline{\nabla_{gH}^2 G_i(x)}}\right) \left|v\right|  \\
 			\text{and }
 			\overline{g}_x^i(v)	:=&  \tfrac{1}{2}\left(\underline{\nabla_{gH} G_i(x)}+\overline{\nabla_{gH} G_i(x)}\right)^\top v+\tfrac{1}{2}\left({\overline{\nabla_{gH} G_i(x)}-\underline{\nabla_{gH} G_i(x)}}\right)^\top \left|v\right|\\
 			&+  \tfrac{1}{4} v^\top\left(\underline{\nabla_{gH}^2 G_i(x)}+\overline{\nabla_{gH}^2 G_i(x)}\right) v+\tfrac{1}{4}\left|v\right|^\top\left({\overline{\nabla_{gH}^2 G_i(x)}-\underline{\nabla_{gH}^2 G_i(x)}}\right) \left|v\right|.
 		\end{aligned}
 	\end{rcases}
 \end{equation}
 \medskip

 Let $\psi:{\mathbb{R}}^m\rightarrow {\mathbb{R}}$ be a function defined by \[\psi(w):=\underset{i=1,2,\ldots,m}{\max}w_i.\] Note that $\psi:{\mathbb{R}}^m\rightarrow {\mathbb{R}}$ is Lipschitz continuous with Lipschitz constant $1$. For a given $x\in U$, let $\phi_x:{\mathbb{R}}^n\rightarrow {\mathbb{R}}^m$ be a function defined by \[\phi_x(v):=\left(\overline{g}_x^1(v),\overline{g}_x^2(v),\ldots,\overline{g}_x^m(v)\right)^\top.\] Then, $\psi\circ \phi_x:{\mathbb{R}}^n\to {\mathbb{R}}$ is a function given by
 \begin{align}\label{psi-phi def}
 	\psi\circ \phi_x\left(v\right):=\underset{i=1,2,\ldots,m}{\max}\overline{g}_x^i\left(v\right).
 \end{align}
	\medskip
\begin{remark}
	\normalfont
	Note that for a given $x\in U$, $\underline{g}_x^i\left(v\right)\leq\overline{g}_x^i\left(v\right)$ for each $i=1,2,\ldots,m$. So, we have \[\psi\circ \phi_x\left(v\right)<0\implies \underset{i=1,2,\ldots,m}{\max}\overline{g}_x^i\left(v\right)<0\implies g_x^i\left(v\right)\prec {\bf 0} \text{ for all }i=1,2,\ldots,m.\]
	However, $\underset{i=1,2,\ldots,m}{\max}\underline{g}_x^i\left(v\right)<0$ does not imply that $g_x^i\left(v\right)\prec {\bf 0}$ for all $i=1,2,\ldots,m$. Due to this reason, we use $\overline{g}_x^i$ instead of $\underline{g}_x^i$ in \eqref{psi-phi def}.
\end{remark}
\medskip 
For $x\in U$, we define $v(x)$, the Newton direction at $x$, as the optimal solution of the unconstrained minimization problem 
\begin{align}\label{unconstrained min}
	\underset{v\in{\mathbb{R}}^n}{\min} \psi\circ\phi_x(v).
\end{align}
Observe that the objective function of the minimization problem \eqref{unconstrained min} is strongly convex. Therefore, the minimization problem \eqref{unconstrained min} has a unique optimal solution. Let $v(x)$ and $\xi(x)$ be the optimal solution and the optimal value of \eqref{unconstrained min} at the point $x$, i.e., 
\begin{align}\label{v(x) and xi(x)} 
	v(x):=\underset{v\in{\mathbb{R}}^n}{\argmin}\:  \psi\circ\phi_x(v) \text{ and } \xi(x):=\underset{v\in{\mathbb{R}}^n}{\min}\: \psi\circ\phi_x(v).
	\end{align}
	\medskip
\begin{remark}
	\normalfont
	If $G$ is a multiobjective real-valued function, then the lower and the upper boundary functions of the IVM $G_i$ are equal, i.e., $G_i=\underline{G}_i=\overline{G}_i$ for all $i=1,2,\ldots,m$. So, we get $\nabla G_i=\underline{\nabla_{gH} G_i}=\overline{\nabla_{gH} G_i}$ and $\nabla^2 G_i=\underline{\nabla_{gH}^2 G_i}=\overline{\nabla_{gH}^2 G_i}$ for all $i=1,2,\ldots,m$. Therefore, in such a particular case, the computation of $v(x)$ in \eqref{v(x) and xi(x)} reduces to 
	\begin{align}\label{Newton direction v(x) in real-valued multiobjective case }
		&\underset{v\in{\mathbb{R}}^n}{\argmin}\left(\underset{i=1,2,\ldots,m}{\max}\nabla G_i(x)^\top v+\tfrac{1}{2}v^\top \nabla^2 G_i(x)v\right),
	\end{align} 
	which is identical to the expression of Newton direction for real-valued multiobjective optimization in \cite{fliege2009newton}. So, \eqref{v(x) and xi(x)} is a true generalization of the conventional Newton method \cite{fliege2009newton} of multiobjective optimization problems.
\end{remark}
	
	\bigskip
	\noindent
Although the problem \eqref{unconstrained min} is a nonsmooth problem, to compute $v(x)$, we reformulate the problem \eqref{unconstrained min} as follows:
\begin{align*}
	 &\underset{v\in{\mathbb{R}}^n}{\min}\underset{i=1,2,\ldots,m}{\max}\overline{g}_x^i\left( v\right)\\
	~\equiv~
	&\underset{v\in{\mathbb{R}}^n}{\min}\biggl(\underset{i=1,2,\ldots,m}{\max}\biggl[\tfrac{1}{2}\left(\underline{\nabla_{gH} G_i(x)}+\overline{\nabla_{gH} G_i(x)}\right)^\top v+\tfrac{1}{2}\left({\overline{\nabla_{gH} G_i(x)}-\underline{\nabla_{gH} G_i(x)}}\right)^\top \left|v\right|\\
	&+  \tfrac{1}{4} v^\top\left(\underline{\nabla_{gH}^2 G_i(x)}+\overline{\nabla_{gH}^2 G_i(x)}\right) v+\tfrac{1}{4}\left|v\right|^\top\left({\overline{\nabla_{gH}^2 G_i(x)}-\underline{\nabla_{gH}^2 G_i(x)}}\right) \left|v\right|\biggr]\biggr)\\
	~\equiv~ &\underset{u,v\in{\mathbb{R}}^n}{\min}\biggl(\underset{i=1,2,\ldots,m}{\max}\biggl[\tfrac{1}{2}\left(\underline{\nabla_{gH} G_i(x)}+\overline{\nabla_{gH} G_i(x)}\right)^\top v+\tfrac{1}{2}\left({\overline{\nabla_{gH} G_i(x)}-\underline{\nabla_{gH} G_i(x)}}\right)^\top u\\
	&+  \tfrac{1}{4} v^\top\left(\underline{\nabla_{gH}^2 G_i(x)}+\overline{\nabla_{gH}^2 G_i(x)}\right) v+\tfrac{1}{4}u^\top\left({\overline{\nabla_{gH}^2 G_i(x)}-\underline{\nabla_{gH}^2 G_i(x)}}\right) u\biggr]\biggr)\\
	&\text{subject to } -u_j\leq v_j \leq u_j, j=1,2,\ldots n,
\end{align*}
that is, 
\begin{equation}\label{equivalent constrained problem}
	\begin{rcases}
\begin{aligned}
	\underset{u,v\in{\mathbb{R}}^n,\tau\in{\mathbb{R}}}{\min}&\tau\\
	 \text{subject to } &\tfrac{1}{2}\left(\underline{\nabla_{gH} G_i(x)}+\overline{\nabla_{gH} G_i(x)}\right)^\top v+\tfrac{1}{2}\left({\overline{\nabla_{gH} G_i(x)}-\underline{\nabla_{gH} G_i(x)}}\right)^\top u\\
	 &+  \tfrac{1}{4} v^\top\left(\underline{\nabla_{gH}^2 G_i(x)}+\overline{\nabla_{gH}^2 G_i(x)}\right) v+\tfrac{1}{4}u^\top\left({\overline{\nabla_{gH}^2 G_i(x)}-\underline{\nabla_{gH}^2 G_i(x)}}\right) u\leq\tau, i=1,2,\ldots m,\\
	&-u_j\leq v_j \leq u_j, j=1,2,\ldots n.
\end{aligned}
\end{rcases}
\end{equation}

\bigskip

\noindent
The computed Newton direction at $x$ by solving \eqref{equivalent constrained problem} has interrelation with Pareto critical points of the MIOP \eqref{minG(x)}. The following result also helps to identify a stopping condition for Pareto critical point by the value of $\left\|v(x)\right\|$ and $\xi(x)$. It indicates that if $\xi(x)=0$ or $\left\|v(x)\right\|=0$, then $x$ is a Pareto critical point. If, however, $\xi(x)\ne0$ or $\left\|v(x)\right\|\ne0$, then $v(x)$ is a descent direction of the objective function of the MIOP \eqref{minG(x)} at $x$.

\medskip
\begin{theorem}\label{descent direction finding theorem}
	Let $v(x)$ and $\xi(x)$ be the optimal solution and the optimal value of \eqref{unconstrained min} at the point $x$, i.e., $v(x):=\underset{v\in{\mathbb{R}}^n}{\argmin}\:  \psi\circ\phi_x(v)$ and $\xi(x):=\underset{v\in{\mathbb{R}}^n}{\min}  \psi\circ\phi_x(v)$. Then, 
	\begin{itemize}
		\item[(i)] For any $x\in U$, $\xi(x)\leq 0$.
		\item[(ii)] The following conditions are equivalent. 
		\begin{itemize}
			\item[a.] The point $x$ is non-Pareto critical.
			\item[b.] $\xi\left(x\right)<0$.
			\item[c.] $v\left(x\right)\neq 0$.
		\end{itemize}
		\item[(iii)] The mapping $x\mapsto v(x)$ is bounded on a compact subset of $U$.
		\item[(iv)] The mapping $x\mapsto \xi(x)$ is continuous in $U$.
	\end{itemize}
\end{theorem}

\medskip
\begin{proof}
 For any $x\in U$, we have
	\begin{align*}
		\xi(x)&\leq \psi\circ\phi_x(0)+\tfrac{1}{2}\|0\|^2=\underset{i=1,2,\ldots,m}{\max}\overline{g}_x^i\left( 0\right)=0.
	\end{align*}
	Therefore, item (i) is true.
	
	\medskip
	\noindent To prove item (ii), we will show that (a) implies (b), (b) implies (c), and (c) implies (a).
	 To show (a) implies (b), let $x$ be a non-Pareto critical point of the MIOP \eqref{minG(x)}. Then, there exists a $v\in{\mathbb{R}}^n$ such that $\nabla_{gH}G_i(x)^\top\odot v \prec {\bf 0}$ for all i=1,2,$\ldots,m.$ So, we have 
	 \[\underset{i=1,2,\ldots,m}{\max}\left[\tfrac{1}{2}\left(\underline{\nabla_{gH} G_i(x)}+\overline{\nabla_{gH} G_i(x)}\right)^\top v+\tfrac{1}{2}\left({\overline{\nabla_{gH} G_i(x)}-\underline{\nabla_{gH} G_i(x)}}\right)^\top \left|v\right|\right]< 0.\] 
	 For all $\theta>0$, we get 
	 \begin{align*}
	 	\xi(x)~\leq~ & \psi\circ\phi_x(\theta v)\\
	 	~=~&\underset{i=1,2,\ldots,m}{\max}\biggl[\tfrac{1}{2}\left(\underline{\nabla_{gH} G_i(x)}+\overline{\nabla_{gH} G_i(x)}\right)^\top \left(\theta v\right)+\tfrac{1}{2}\left({\overline{\nabla_{gH} G_i(x)}-\underline{\nabla_{gH} G_i(x)}}\right)^\top \left|\theta v\right|\\
	 	&+  \tfrac{1}{4} \left(\theta v\right)^\top\left(\underline{\nabla_{gH}^2 G_i(x)}+\overline{\nabla_{gH}^2 G_i(x)}\right) \left(\theta v\right)+\tfrac{1}{4}\left|\theta v\right|^\top\left({\overline{\nabla_{gH}^2 G_i(x)}-\underline{\nabla_{gH}^2 G_i(x)}}\right) \left|\theta v\right|\biggr] \\
	 	~=~&\theta\biggl(\underset{i=1,2,\ldots,m}{\max}\biggl[\tfrac{1}{2}\left(\underline{\nabla_{gH} G_i(x)}+\overline{\nabla_{gH} G_i(x)}\right)^\top v+\tfrac{1}{2}\left({\overline{\nabla_{gH} G_i(x)}-\underline{\nabla_{gH} G_i(x)}}\right)^\top \left| v\right|\\
	 	&+ \theta\left\{\tfrac{1}{4} v^\top\left(\underline{\nabla_{gH}^2 G_i(x)}+\overline{\nabla_{gH}^2 G_i(x)}\right) v+\tfrac{1}{4}\left| v\right|^\top\left({\overline{\nabla_{gH}^2 G_i(x)}-\underline{\nabla_{gH}^2 G_i(x)}}\right) \left| v\right|\right\}\biggr]\biggr).
	 \end{align*}
	Let us choose $\theta$ that satisfies the following condition \[0<\theta<-\tfrac{\tfrac{1}{2}\left(\underline{\nabla_{gH} G_i(x)}+\overline{\nabla_{gH} G_i(x)}\right)^\top v+\tfrac{1}{2}\left({\overline{\nabla_{gH} G_i(x)}-\underline{\nabla_{gH} G_i(x)}}\right)^\top \left| v\right|}{\tfrac{1}{4} v^\top\left(\underline{\nabla_{gH}^2 G_i(x)}+\overline{\nabla_{gH}^2 G_i(x)}\right) v+\tfrac{1}{4}\left| v\right|^\top\left({\overline{\nabla_{gH}^2 G_i(x)}-\underline{\nabla_{gH}^2 G_i(x)}}\right) \left| v\right|} \text{ for all }i=1,2,\ldots,m.\]
	Therefore, we get $\xi(x)<0.$ 
	
	 \medskip
	 \noindent
	 To show (b) implies (c), let $\xi(x)<0.$ If possible, let $v(x)=0.$ This implies $\xi(x)=0,$ which contradicts $\xi(x)<0.$ Therefore, $v(x)\neq 0.$
	 
	 \medskip
	 \noindent
	 Let us now show (c) implies (a). Since for all $i=1,2,\ldots,m$, $\nabla_{gH}^2 G_i(x)\succ {\bf 0}$, we have 
	 \begin{align*}
	 	&\tfrac{1}{2}\left(\underline{\nabla_{gH} G_i(x)}+\overline{\nabla_{gH} G_i(x)}\right)^\top v(x)+\tfrac{1}{2}\left({\overline{\nabla_{gH} G_i(x)}-\underline{\nabla_{gH} G_i(x)}}\right)^\top \left| v(x)\right|\\
	 	~<~&\tfrac{1}{2}\left(\underline{\nabla_{gH} G_i(x)}+\overline{\nabla_{gH} G_i(x)}\right)^\top v(x)+\tfrac{1}{2}\left({\overline{\nabla_{gH} G_i(x)}-\underline{\nabla_{gH} G_i(x)}}\right)^\top \left| v(x)\right|\\
	 	&+\tfrac{1}{4} v(x)^\top\left(\underline{\nabla_{gH}^2 G_i(x)}+\overline{\nabla_{gH}^2 G_i(x)}\right) v(x)+\tfrac{1}{4}\left| v(x)\right|^\top\left({\overline{\nabla_{gH}^2 G_i(x)}-\underline{\nabla_{gH}^2 G_i(x)}}\right) \left| v(x)\right|=\xi(x)<0.
	 \end{align*}
	 Therefore, we get $\nabla_{gH}G_i(x)^\top \odot v(x)\prec {\bf 0}$ for all $i=1,2,\ldots,m$. This implies $v(x)$ is a descent direction, and consequently, $x$ is a non-Pareto critical point.
	
	\medskip
	\noindent
	To prove item (iii), let $\mathcal{S}\subset U$ be compact. For all $x\in\mathcal{S}$ and for all $i=1,2,\ldots,m$, we have
	\begin{align*}
		 &-\tfrac{1}{2}\left\lVert\underline{\nabla_{gH} G_i(x)}+\overline{\nabla_{gH} G_i(x)}\right\rVert\|v(x)\|+\tfrac{\lambda_{\min}}{4}\|v(x)\|^2\\
		 ~\leq~ & \tfrac{1}{2}\left(\underline{\nabla_{gH} G_i(x)}+\overline{\nabla_{gH} G_i(x)}\right)^\top v(x)+\tfrac{1}{4} v(x)^\top\left(\underline{\nabla_{gH}^2 G_i(x)}+\overline{\nabla_{gH}^2 G_i(x)}\right) v(x)\\
		~\leq~ & \tfrac{1}{2}\left(\underline{\nabla_{gH} G_i(x)}+\overline{\nabla_{gH} G_i(x)}\right)^\top v(x)+\tfrac{1}{2}\left({\overline{\nabla_{gH} G_i(x)}-\underline{\nabla_{gH} G_i(x)}}\right)^\top \left| v(x)\right|\\
		&+\tfrac{1}{4} v(x)^\top\left(\underline{\nabla_{gH}^2 G_i(x)}+\overline{\nabla_{gH}^2 G_i(x)}\right) v(x)+\tfrac{1}{4}\left| v(x)\right|^\top\left({\overline{\nabla_{gH}^2 G_i(x)}-\underline{\nabla_{gH}^2 G_i(x)}}\right) \left| v(x)\right|=\xi(x)\leq0,
	\end{align*}
	 where $\lambda_{\min}$ is the minimum eigenvalue of $\underline{\nabla_{gH}^2 G_i(x)}+\overline{\nabla_{gH}^2 G_i(x)}$.
	 Therefore, we get 
	 \begin{align}\label{v bounded}
	 	\|v(x)\|\leq \tfrac{2}{\lambda_{\min}}\left\lVert\underline{\nabla_{gH}G_i(x)}+\overline{\nabla_{gH} G_i(x)}\right\rVert\leq \tfrac{2}{\lambda_{\min}}\left( \left\lVert\underline{\nabla_{gH}G_i(x)}\right\rVert+\left\lVert\overline{\nabla_{gH} G_i(x)}\right\rVert\right).
	 \end{align}
	  Since $G_i$ is a $gH$-continuously differentiable IVM, $\nabla_{gH}G_i$ is $gH$-continuous. Then, by Lemma \ref{convexity, continuity with boundary functions}, its boundary functions  $\underline{\nabla_{gH}G_i}$ and $\overline{\nabla_{gH}G_i}$ are continuous on the compact set $\mathcal{S}$. So, $\underline{\nabla_{gH}G_i}$ and $\overline{\nabla_{gH}G_i}$ are bounded on the compact set $\mathcal{S}$. Therefore, there exists an $M>0$ such that \[\|v(x)\|\leq M \text{ for all }x \in \mathcal{S}.\]
	   Hence, $v$ is bounded on the compact set $\mathcal{S}$.
	 
	 \medskip
	 \noindent
	  To prove item (iv), let $\tilde{x}\in U$ be arbitrary and $\left\{x^k\right\}$ be a sequence such that $x^k\to \tilde{x}$ as $k\to\infty$. We will prove that $\underset{k\to\infty}{\lim}\xi\left(x^k\right)=\xi(\tilde{x})$.
	  By the optimality of $v\left(x^k\right)$, we have for all $k$,
	 \begin{align*}
	 	\xi\left(x^k\right)=&\:\psi\circ\phi_{x^k}\left( v\left(x^k\right)\right)\leq \psi\circ\phi_{x^k}\left( v(\tilde{x})\right)\\
	 	\overset{\eqref{psi-phi def}}{\implies}\xi\left(x^k\right)\leq&\:
	 	\underset{i=1,2,\ldots,m}{\max} \overline{g}_{x^k}^i\left(v\left(\tilde{x}\right)\right)\\
	 	\overset{\eqref{gxi-lower-upper}}{\implies}  \xi\left(x^k\right)\leq &\: \underset{i=1,2,\ldots,m}{\max}\biggl(\tfrac{1}{2}\left(\underline{\nabla_{gH} G_i\left(x^k\right)}+\overline{\nabla_{gH} G_i\left(x^k\right)}\right)^\top v\left(\tilde{x}\right)\\
	 	&+\tfrac{1}{2}\left({\overline{\nabla_{gH} G_i\left(x^k\right)}-\underline{\nabla_{gH} G_i\left(x^k\right)}}\right)^\top \left|v\left(\tilde{x}\right)\right|\\
	 	&+\tfrac{1}{4} v\left(\tilde{x}\right)^\top\left(\underline{\nabla_{gH}^2 G_i\left(x^k\right)}+\overline{\nabla_{gH}^2 G_i\left(x^k\right)}\right) v\left(\tilde{x}\right)\\
	 	&+\tfrac{1}{4} \left|v\left(\tilde{x}\right)\right|^\top\left({\overline{\nabla_{gH}^2 G_i\left(x^k\right)}-\underline{\nabla_{gH}^2 G_i\left(x^k\right)}}\right) \left|v\left(\tilde{x}\right)\right|\biggr).
	 \end{align*}
	 Since $G\in C_{gH}^2\left(U,I\left({\mathbb{R}}\right)^m\right)$ and $x^k\to \tilde{x}$, for all $i=1,2,\ldots,m$, we get 
	 \begin{equation}\label{grad and hess continuity result}
	 	\begin{rcases}
	 		\begin{aligned}
	 			&\underline{\nabla_{gH} G_i\left(x^k\right)}\to\underline{\nabla_{gH} G_i\left(\tilde{x}\right)} \text{ and }\overline{\nabla_{gH} G_i\left(x^k\right)}\to\overline{\nabla_{gH} G_i\left(\tilde{x}\right)}\\
	 			&\underline{\nabla_{gH}^2 G_i\left(x^k\right)}\to\underline{\nabla_{gH}^2 G_i\left(\tilde{x}\right)} \text{ and }\overline{\nabla_{gH}^2 G_i\left(x^k\right)}\to\overline{\nabla_{gH}^2 G_i\left(\tilde{x}\right)}
	 		\end{aligned}
	 	\end{rcases}
	 	\text{ as } k\to\infty.
	 \end{equation}
	 Therefore, we get 
	  \begin{align*}
	 	\underset{k\to\infty}{\lim\sup }\:\xi\left(x^k\right)\leq &\: \underset{i=1,2,\ldots,m}{\max}\biggl(\tfrac{1}{2}\left(\underline{\nabla_{gH} G_i\left(\tilde{x}\right)}+\overline{\nabla_{gH} G_i\left(\tilde{x}\right)}\right)^\top v\left(\tilde{x}\right)\\
	 	&+\tfrac{1}{2}\left({\overline{\nabla_{gH} G_i\left(\tilde{x}\right)}-\underline{\nabla_{gH} G_i\left(\tilde{x}\right)}}\right)^\top \left|v\left(\tilde{x}\right)\right|\\
	 	&+\tfrac{1}{4} v\left(\tilde{x}\right)^\top\left(\underline{\nabla_{gH}^2 G_i\left(\tilde{x}\right)}+\overline{\nabla_{gH}^2 G_i\left(\tilde{x}\right)}\right) v\left(\tilde{x}\right)\\
	 	&+\tfrac{1}{4} \left|v\left(\tilde{x}\right)\right|^\top\left({\overline{\nabla_{gH}^2 G_i\left(\tilde{x}\right)}-\underline{\nabla_{gH}^2 G_i\left(\tilde{x}\right)}}\right) \left|v\left(\tilde{x}\right)\right|
	 	\biggr),
	 \end{align*}
 which implies
	 \begin{align}\label{proof of cont eq1}
	 	\underset{k\to\infty}{\lim\sup }\:\xi\left(x^k\right)\leq \psi\circ\phi_{\tilde{x}}\left( v(\tilde{x})\right)=\xi(\tilde{x}).
	 \end{align}
	 On the other hand, we have 
	 \begin{align*}
	 \xi(\tilde{x})=\underset{v\in{\mathbb{R}}^n}{\min} \psi\circ\phi_{\tilde{x}}(v)
	 \leq \psi\circ\phi_{\tilde{x}}\left(v\left(x^k\right)\right).
	\end{align*}
	Therefore, we get 
	\begin{align*}
		\xi(\tilde{x})&\leq \underset{k\to\infty}{\lim\inf }\: \psi\circ\phi_{\tilde{x}}\left(v\left(x^k\right)\right)= \underset{k\to\infty}{\lim\inf }\:\left[ \psi\circ\phi_{\tilde{x}}\left(v\left(x^k\right)\right)+\psi\circ\phi_{x^k}\left(v\left(x^k\right)\right)-\psi\circ\phi_{x^k}\left(v\left(x^k\right)\right)\right]\\
		&= \underset{k\to\infty}{\lim\inf }\:\left[\xi\left(x^k\right)+ \psi\circ\phi_{\tilde{x}}\left(v\left(x^k\right)\right)-\psi\circ\phi_{x^k}\left(v\left(x^k\right)\right)\right]\\
		&\leq \underset{k\to\infty}{\lim\inf }\:\left(\xi\left(x^k\right)+\left\lVert\phi_{x^k}\left(v\left(x^k\right)\right)-\phi_{\tilde{x}}\left(v\left(x^k\right)\right)\right\rVert\right) \quad\left[\text{Due to Lipschitz continuity of } \psi\right]\\
	\end{align*}
	We have \[\phi_{x^k}\left(v\left(x^k\right)\right)=\left(\overline{g}_{x^k}^1\left(v\left(x^k\right)\right),\overline{g}_{x^k}^2\left(v\left(x^k\right)\right),\ldots,\overline{g}_{x^k}^m\left(v\left(x^k\right)\right)\right)^\top\]
	and 
	\[\phi_{\tilde{x}}\left(v\left(x^k\right)\right)=\left(\overline{g}_{\tilde{x}}^1\left(v\left(x^k\right)\right),\overline{g}_{\tilde{x}}^2\left(v\left(x^k\right)\right),\ldots,\overline{g}_{\tilde{x}}^m\left(v\left(x^k\right)\right)\right)^\top.\]
	For all $i=1,2,\ldots,m$, from \eqref{gxi-lower-upper} we have that
	\begin{align*}
		\overline{g}_{x^k}^i\left(v\left(x^k\right)\right)~=~&\tfrac{1}{2}\left(\underline{\nabla_{gH} G_i\left(x^k\right)}+\overline{\nabla_{gH} G_i\left(x^k\right)}\right)^\top v\left(x^k\right)\\
		&+\tfrac{1}{2}\left({\overline{\nabla_{gH} G_i\left(x^k\right)}-\underline{\nabla_{gH} G_i\left(x^k\right)}}\right)^\top \left|v\left(x^k\right)\right|\\
		&+\tfrac{1}{4} v\left(x^k\right)^\top\left(\underline{\nabla_{gH}^2 G_i\left(x^k\right)}+\overline{\nabla_{gH}^2 G_i\left(x^k\right)}\right) v\left(x^k\right)\\
		&+\tfrac{1}{4} \left|v\left(x^k\right)\right|^\top\left({\overline{\nabla_{gH}^2 G_i\left(x^k\right)}-\underline{\nabla_{gH}^2 G_i\left(x^k\right)}}\right) \left|v\left(x^k\right)\right|
	\end{align*}
	and 
	\begin{align*}
		\overline{g}_{\tilde{x}}^i\left(v\left(x^k\right)\right)~=~&\tfrac{1}{2}\left(\underline{\nabla_{gH} G_i\left(\tilde{x}\right)}+\overline{\nabla_{gH} G_i\left(\tilde{x}\right)}\right)^\top v\left(x^k\right)\\
		&+\tfrac{1}{2}\left({\overline{\nabla_{gH} G_i\left(\tilde{x}\right)}-\underline{\nabla_{gH} G_i\left(\tilde{x}\right)}}\right)^\top \left|v\left(x^k\right)\right|\\
		&+\tfrac{1}{4} v\left(x^k\right)^\top\left(\underline{\nabla_{gH}^2 G_i\left(\tilde{x}\right)}+\overline{\nabla_{gH}^2 G_i\left(\tilde{x}\right)}\right) v\left(x^k\right)\\
		&+\tfrac{1}{4} \left|v\left(x^k\right)\right|^\top\left({\overline{\nabla_{gH}^2 G_i\left(\tilde{x}\right)}-\underline{\nabla_{gH}^2 G_i\left(\tilde{x}\right)}}\right) \left|v\left(x^k\right)\right|.
	\end{align*}
	From \eqref{v bounded}, we get $\|v(x)\|\leq \tfrac{2}{\lambda_{\min}}\left( \left\lVert\underline{\nabla_{gH}G_i(x)}\right\rVert+\left\lVert\overline{\nabla_{gH} G_i(x)}\right\rVert\right)$.
	Since $\underline{\nabla_{gH}G_i}$ and $\overline{\nabla_{gH}G_i}$ are continuous and $x^k\to \tilde{x}$ as $k\to\infty$, we conclude that $\left\{v\left(x^k\right)\right\}$ is bounded. Then, letting $k\to\infty$, using \eqref{grad and hess continuity result}, we get $\overline{g}_{x^k}^i\left(v\left(x^k\right)\right)-\overline{g}_{\tilde{x}}^i\left(v\left(x^k\right)\right)\to 0$. Consequently, we get $\phi_{x^k}\left(v\left(x^k\right)\right)-\phi_{\tilde{x}}\left(v\left(x^k\right)\right)\to 0$ as $k\to\infty$. Therefore, we obtain that 
	\begin{align}\label{proof of cont eq2}
		\xi(\tilde{x})\leq \underset{k\to\infty}{\lim\inf }\:\xi\left(x^k\right).
	\end{align}
	From \eqref{proof of cont eq1} and \eqref{proof of cont eq2}, we get $\underset{k\to\infty}{\lim\sup }\:\xi\left(x^k\right)\leq\xi(\tilde{x})\leq \underset{k\to\infty}{\lim\inf }\:\xi\left(x^k\right)$, and this implies that $\xi$ is continuous.
\end{proof}

\medskip
\subsection{Step Length}
Suppose that we have a direction vector $v\in{\mathbb{R}}^n$ such that $\nabla_{gH} G_i(x)^\top\odot v\prec{\bf 0}$ for all $i=1,2,\ldots,m$. To compute the step length $t>0$, we use the Armijo-like rule. Let $\sigma\in (0,1)$ be a predefined constant. The Armijo-like condition for $t$ to be acceptable is 
\[G_i(x+tv)\preceq G_i(x)\oplus\left[\sigma t,\sigma t\right]\odot \xi(x) \text{ for all }i=1,2,\ldots,m.\] 
More precisely, the acceptance condition of $t$ is 
\begin{equation}\label{Armijo condition}
	\begin{rcases}
\begin{aligned}
	&\underline{G}_i(x+tv)\leq \underline{G}_i(x)+\sigma t \:\xi(x)\\
	\text{and }
	&\overline{G}_i(x+tv)\leq \overline{G}_i(x)+\sigma t \: \xi(x) \text{ for all }i=1,2,\ldots,m.
\end{aligned}
\end{rcases}
\end{equation}
 For computation of the step length $t$, we initially set $t=1$, and while the conditions given in \eqref{Armijo condition} are not satisfied, we set $t:=\eta t,$ where $\eta\in(0,1)$ is a reduction factor. Now, we proceed the existence guarantee of the step length in the following result.
 
 \medskip
 \begin{theorem}\label{steplength theorem}
 	Assume that $G\in C_{gH}^2\left(U,I\left({\mathbb{R}}\right)^m\right)$ and $\sigma\in(0,1)$. If $\nabla_{gH} G_i(x)^\top\odot v\prec{\bf 0}$ and $\nabla_{gH}^2 G_i(x)\succ {\bf 0}$ for all $i=1,2,\ldots,m$, then there exists a $\delta>0$ such that for all $i=1,2,\ldots,m$, \[G_i(x+tv)\prec G_i(x)\oplus\left[\sigma t,\sigma t\right]\odot \xi(x)\text{ for any }t\in(0,\delta].\]
 \end{theorem}
 
 \medskip
 \begin{proof}
 Since $G\in C_{gH}^2\left(U,I\left({\mathbb{R}}\right)^m\right)$, $\sigma\in(0,1)$, and $\nabla_{gH} G_i(x)^\top\odot v\prec{\bf 0}$, from Lemma \ref{linear IVM lemma}, we have \[\underset{t\to0}{\lim} \tfrac{1}{t}\odot\left[G_i(x+tv)\ominus_{gH} G_i(x)\right]:=\nabla_{gH} G_i(x)^\top\odot v\prec \sigma\odot\nabla_{gH} G_i(x)^\top\odot v \text{ for all } i=1,2,\ldots,m.\]
 Therefore, there exists a $\delta>0$ such that for all $t\in(0,\delta]$ and for all $i=1,2,\ldots,m$, we have
 \[G_i(x+tv)\ominus_{gH} G_i(x)\prec \sigma t\odot\nabla_{gH} G_i(x)^\top\odot v 
 \implies G_i(x+tv)\prec G_i(x)\oplus \sigma t\odot\nabla_{gH} G_i(x)^\top\odot v.\]
 Since $\nabla_{gH}^2 G_i(x)\succ {\bf 0}$ for all $i=1,2,\ldots,m$, we get 
 \begin{align*}
 	G_i(x+tv)\ominus_{gH} G_i(x)\prec \sigma t\odot\nabla_{gH} G_i(x)^\top\odot v \oplus \tfrac{\sigma t}{2}\odot v^\top \odot\nabla_{gH}^2 G_i(x)\odot v\prec\left[\sigma t,\sigma t\right]\odot \xi(x),
 \end{align*}
 which concludes the proof.
 \end{proof}
 
\medskip
 
Next, we present a step-wise algorithm of the Newton method for identifying Pareto critical points of the MIOP \eqref{minG(x)}.
\medskip

\begin{algorithm}[H]
	\caption{Newton method to find Pareto critical points of the MIOP \eqref{minG(x)} \label{Algorithm}}
	\begin{enumerate}[\bf{Step} 1 ]
		
		\item (Inputs)\\
		Provide $G :=  \left( G_1, G_2, \ldots, G_m \right)^\top$, where $G_1,G_2,\ldots, G_m$  are twice $gH$-continuously differentiable IVMs. Also, provide the domain $U:=\left\{x\in{\mathbb{R}}^n: lb\leq x \leq ub\right\}\subseteq{\mathbb{R}}^n$ of the decision variables for the MIOP \eqref{minG(x)}.\\
		\item (Initialization)\\
		Choose the step length reduction factor $\eta\in(0,1)$, line search parameter $\sigma\in(0,1)$ for the Armijo-like rule \eqref{Armijo condition}, and a random point $x^0 \in {\mathbb{R}}^n$ from the domain $U$. Provide the tolerance level $\epsilon > 0$. Set $k=0$.\\

		\item (Computation of $gH$-gradient and $gH$-Hessian at the point $x^k$)\\
		For all $i=1,2,\ldots,m$, compute \[\nabla_{gH} G_i\left(x^k\right):=\left[\underline{\nabla_{gH}G_i\left(x^k\right)},\overline{\nabla_{gH}G_i\left(x^k\right)}\right] \text{ and }\nabla_{gH}^2 G_i\left(x^k\right):=\left[\underline{\nabla_{gH}^2G_i\left(x^k\right)},\overline{\nabla_{gH}^2G_i\left(x^k\right)}\right].\]
		\\
		\item (Computation of a Newton direction at the point $x^k$)\\
		Compute the optimal solution $v\left(x^k\right)$ and the optimal value $\xi\left(x^k\right)$ of the unconstrained minimization problem \eqref{unconstrained min}, i.e.,
		\[
		v\left(x^k\right):=\underset{v\in{\mathbb{R}}^n}{\argmin}\: \psi\circ\phi_{x^k}(v)\text{ and }
		\xi\left(x^k\right):=\underset{v\in{\mathbb{R}}^n}{\min}\: \psi\circ\phi_{x^k}(v).
		\]

		\item
		(Stopping condition)\\
		If $\xi\left(x^k\right)>- \epsilon$, Stop and return $x^k$ as a Pareto critical point.\\
		Otherwise, go to {\bf Step 6}.\\

		\item (Computation of step length)\\
	Set $t_{k} \gets 1$.
	Keep reducing the step length by $t_{k}:=\eta t_{k}$ until \eqref{Armijo condition} is satisfied.\\
	\item (Update the iterative point)\\
		Update $x^{k+1} \gets x^k + t_k v\left(x^k\right)$, $k \gets k + 1$ and go to {\bf Step 3}.
		
	\end{enumerate}
	
\end{algorithm}
The well-definedness of Algorithm \ref{Algorithm} depends on {\bf Step 4} and {\bf Step 6}. Since the objective function of the minimization problem \eqref{unconstrained min} is strongly convex, it has a unique optimal solution. This guarantees the existence of the Newton direction $v\left(x^k\right)$, and consequently, this implies the well-definedness of {\bf Step 4}. In addition, Theorem \ref{steplength theorem} ensures the existence of the step length $t_k$, which guarantees the well-definedness of {\bf Step 6}. Hence, Algorithm \ref{Algorithm} is well-defined. 

\medskip

Now, we study the scaling effect of the Newton method for the MIOP \eqref{minG(x)}. The following result shows that the iteration scheme of the Newton method for the MIOP \eqref{minG(x)} is scaling independent of variable $x$. 
\medskip
\begin{proposition}\label{scaling of variable}
	Assume that $t_k=1$ for all $k\in{\mathbb{N}}$. If $T$ is a nonsingular matrix of order $n\times n$ such that $\hat{x}:=Tx$, then the iterative scheme of Algorithm \ref{Algorithm} $x^{k+1}:=x^k+v\left(x^k\right)$ is transformed to the iterative scheme $\hat{x}^{k+1}:=\hat{x}^k+\hat{v}\left(\hat{x}^k\right)$, where $\hat{v}\left(\hat{x}^k\right):=Tv\left(x^k\right)$. 
\end{proposition}
\medskip

\begin{proof}
	By the transformation given by $\hat{x}:=Tx$, the multiobjective IVM $G:=\left(G_1,G_2,\ldots,G_m\right)^\top$ is transformed to $\hat{G}:=\left(\hat{G}_1,\hat{G}_2,\ldots,\hat{G}_m\right)^\top$, i.e., $\hat{G}\left(\hat{x}\right)=G(x)=G\left(T^{-1}\hat{x}\right)$. 
	So, for all $i=1,2,\ldots,m$, we have 
	\begin{equation}\label{variable scaling eq.1}
				\nabla_{gH\hat{x}}\hat{G}_i\left(\hat{x}\right):=\left(T^{-1}\right)^\top \odot \nabla_{gH}G_i(x)
				\text{ and }  \nabla_{gH\hat{x}}^2\hat{G}_i\left(\hat{x}\right):=\left(T^{-1}\right)^\top \odot \nabla_{gH}^2G_i(x)\odot T^{-1}.
	\end{equation}
	Therefore, from \eqref{gxi}, we get 
	\begin{align*}
		g_{\hat{x}}^i(v):=~&\nabla_{gH\hat{x}} \hat{G}_i\left(\hat{x}\right)^\top\odot v\oplus \tfrac{1}{2}\odot v^\top \odot \nabla_{gH\hat{x}}^2 \hat{G}_i\left(\hat{x}\right)\odot v\\
		\overset{\eqref{variable scaling eq.1}}{=}~& \nabla_{gH} G_i\left(x\right)^\top\odot \left(T^{-1}v\right)\oplus \tfrac{1}{2}\odot \left(T^{-1}v\right)^\top \odot \nabla_{gH}^2 G_i\left(\hat{x}\right)\odot \left(T^{-1}v\right).
	\end{align*}
	We have $\hat{x}^{k+1}:=\hat{x}^k+\hat{v}\left(\hat{x}^k\right)$, where $\hat{v}\left(\hat{x}^k\right)$ is calculated by
	 \begin{align*}
	 	\hat{v}\left(\hat{x}^k\right):=~&\underset{v\in{\mathbb{R}}^n}{\argmin} \left(\underset{i=1,2,\ldots,m}{\max}\left\{\text{upper boundary function of } g_{\hat{x}}^i(v)\right\}\right)\\
	 	=~&\underset{v\in{\mathbb{R}}^n}{\argmin} \biggl(\underset{i=1,2,\ldots,m}{\max}\biggl[\tfrac{1}{2}\left(\underline{\nabla_{gH} G_i\left(x^k\right)}+\overline{\nabla_{gH} G_i\left(x^k\right)}\right)^\top \left(T^{-1}v\right)\\
	 	&\hspace{2.8cm}+\tfrac{1}{2}\left({\overline{\nabla_{gH} G_i\left(x^k\right)}-\underline{\nabla_{gH} G_i\left(x^k\right)}}\right)^\top \left|T^{-1}v\right|\\
	 	&\hspace{2.8cm} +  \tfrac{1}{4} \left(T^{-1}v\right)^\top\left(\underline{\nabla_{gH}^2 G_i\left(x^k\right)}+\overline{\nabla_{gH}^2 G_i\left(x^k\right)}\right) T^{-1}v\\
	 	&\hspace{2.8cm}+\tfrac{1}{4}\left|T^{-1}v\right|^\top\left({\overline{\nabla_{gH}^2 G_i\left(x^k\right)}-\underline{\nabla_{gH}^2 G_i\left(x^k\right)}}\right) \left|T^{-1}v\right|\biggr]\biggr)\\
	 	=~& Tv\left(x^k\right).
	 \end{align*}
	 Therefore, we get $\hat{x}^{k+1}:=\hat{x}^k+\hat{v}\left(\hat{x}^k\right)\Longleftrightarrow x^{k+1}:=x^k+v\left(x^k\right)$, and this completes the proof. 
\end{proof}

\section{Convergence Analysis}\label{Convergence analysis}
In this section, we provide convergence analysis of Algorithm \ref{Algorithm}. If Algorithm \ref{Algorithm} gives the output after a finite number of iterations, then from {\bf Step 5}, it is obvious that the output is a Pareto critical point. We assume that corresponding to a given initial point $x^0$, an infinite sequence $\left\{x^k\right\}$ is generated by Algorithm \ref{Algorithm} for which $\xi\left(x^k\right)\neq 0$ for all $k\in\mathbb{N}$. Next result shows that all subsequential limits of $\left\{x^k\right\}$ are Pareto critical points of the MIOP \eqref{minG(x)}.

\medskip
\begin{theorem}\label{Convergence theorem}
Every accumulation point of the sequence $\left\{x^k\right\}$ generated by the Algorithm \ref{Algorithm} is a Pareto critical point of the MIOP \eqref{minG(x)}. Further, if the level set $$L_0 := \left\{x\in{\mathbb{R}}^n:G_i(x)\preceq G_i(x^0),i=1,2,\ldots m\right\}$$ is bounded, then the sequence $\left\{x^k\right\}$ remains bounded, and it has at least one accumulation point.
\end{theorem}

\medskip
\begin{proof}
	Let $\bar{x}$ be one of the accumulation points of $\left\{x^k\right\}$. We will prove that $\bar{x}$ is a Pareto critical point of the MIOP \eqref{minG(x)}. Let $v(\bar{x})$ and $\xi(\bar{x})$ be the optimal solution and the optimal value of \eqref{unconstrained min} at the point $\bar{x}$, i.e., $v(\bar{x}):=\underset{v\in{\mathbb{R}}^n}{\argmin}\:  \psi\circ\phi_{\bar{x}}(v)$ and $\xi(\bar{x}):=\underset{v\in{\mathbb{R}}^n}{\min}  \psi\circ\phi_{\bar{x}}(v)$. So, according to Theorem \ref{descent direction finding theorem}, $\bar{x}$ will be a Pareto critical point of the MIOP \eqref{minG(x)} if $\xi\left(\bar{x}\right)=0$. Therefore, our aim is to show that $\xi\left(\bar{x}\right)=0$. For all $k\in{\mathbb{N}}$ and for all $i=1,2,\ldots,m$, we have $G_i\left(x^{k+1}\right)\preceq G_i\left(x^{k}\right)$ and $\underset{k\to\infty}{\lim}G_i\left(x^{k}\right)=G_i\left(\bar{x}\right)$. Therefore, we have
	\begin{align*}
		 & \underset{k\to\infty}{\lim}\underline{G}_i\left(x^{k}\right)=\underline{G}_i\left(\bar{x}\right)\text{ and } \underset{k\to\infty}{\lim} \overline{G}_i\left(x^{k}\right)=\overline{G}_i\left(\bar{x}\right)\\
		\implies & \underset{k\to\infty}{\lim}\left|\underline{G}_i\left(x^{k+1}\right)-\underline{G}_i\left(x^k\right)\right|=0\text{ and } \underset{k\to\infty}{\lim} \left|\overline{G}_i\left(x^{k+1}\right)-\overline{G}_i\left(x^k\right)\right|=0,
	\end{align*}
	which implies
	\begin{align}\label{IR difference 0}
		 \underset{k\to\infty}{\lim}\left\lVert{G_i}\left(x^{k+1}\right)\ominus_{gH}{G_i}\left(x^k\right)\right\rVert_{\it I(\mathbb{R})}=0.
	\end{align}
	As $G_i\left(x^k\right)\ominus_{gH}G_i\left(x^{k+1}\right)\succeq(-1)\odot\left[\sigma t_k,\sigma t_k\right]\odot \xi\left(x^k\right)\succeq{\bf 0}$ for all $i=1,2,\ldots,m$, we must have from \eqref{IR difference 0} that \[\underset{k\to\infty}{\lim}t_k\: \xi\left(x^k\right)=0.\]
	 Note that $t_k\in(0,1]$ for all $k\in{\mathbb{N}}$. So, we have the following two cases: $\underset{k\to\infty}{\lim\sup}\:t_k>0$ or $\underset{k\to\infty}{\lim\sup}\:t_k=0$. 
	\begin{description}
	\item[{\bf Case 1.}] Let $\underset{k\to\infty}{\lim\sup}\:t_k>0$. Then, there exists a subsequence $\left\{x^{k_r}\right\}$ converging to $\bar{x}$ and $\bar{t}>0$ such that $\underset{r\to\infty}{\lim}\:t_{k_r}=\bar{t}$. So, we have 
	$\underset{r\to\infty}{\lim}\:\xi\left(x^{k_r}\right)=0.$
	Since $\xi$ is continuous, we get $\xi\left(\bar{x}\right)=0.$
	Therefore, by Theorem \ref{descent direction finding theorem}, $\bar{x}$ is a Pareto critical point of the MIOP \eqref{minG(x)}.
	\medskip
	\item[{\bf Case 2.}] Let $\underset{k\to\infty}{\lim\sup}\:t_k=0$. From \eqref{v bounded}, we have \[\|v\left(x^k\right)\|\leq\tfrac{2}{\lambda_{\min}}\left( \left\lVert\underline{\nabla_{gH}G_i\left(x^k\right)}\right\rVert+\left\lVert\overline{\nabla_{gH} G_i\left(x^k\right)}\right\rVert\right).\] Since $\underline{\nabla_{gH}G_i}$ and $\overline{\nabla_{gH}G_i}$ are continuous and $\underset{k\to\infty}{\lim}\:x^k=\bar{x}$, the sequence $\left\{v\left(x^k\right)\right\}$ is bounded. Consequently, the sequence $\left\{v\left(x^k\right)\right\}$ has a convergent subsequence. Since $\underset{k\to\infty}{\lim}\:x^k=\bar{x}$,   $\underset{k\to\infty}{\lim\sup}\:t_k=0$, and the sequence $\left\{v\left(x^k\right)\right\}$ has a convergent subsequence, we can take subsequences $\left\{x^{k_r}\right\},\left\{v\left(x^{k_r}\right)\right\}$, and $\left\{t_{k_r}\right\}$ converging to $\bar{x},\bar{v}$, and $0$, respectively. For all $r\in{\mathbb{N}}$, we have \begin{align}\label{xi-kr-less}
	\underset{i=1,2,\ldots,m}{\max}\:\overline{g}_{x^{k_r}}^i\left(v\left(x^{k_r}\right)\right)=\psi\circ\phi_{x^{k_r}}\left(v\left(x^{k_r}\right)\right)=\xi\left(x^{k_r}\right)<0.
\end{align}
For all $i=1,2,\ldots,m$ from \eqref{gxi-lower-upper}, we have 
\begin{align*}
\overline{g}_{x^{k_r}}^i\left(v\left(x^{k_r}\right)\right) ~=~ & \tfrac{1}{2}\left(\underline{\nabla_{gH} G_i\left(x^{k_r}\right)}+\overline{\nabla_{gH} G_i\left(x^{k_r}\right)}\right)^\top v\left(x^{k_r}\right) \\ 
& + \tfrac{1}{2}\left({\overline{\nabla_{gH} G_i\left(x^{k_r}\right)}-\underline{\nabla_{gH} G_i\left(x^{k_r}\right)}}\right)^\top \left|v\left(x^{k_r}\right)\right|\\
+& \tfrac{1}{4}v\left(x^{k_r}\right)^\top\left(\underline{\nabla_{gH}^2 G_i\left(x^{k_r}\right)}+\overline{\nabla_{gH}^2 G_i\left(x^{k_r}\right)}\right) v\left(x^{k_r}\right) \\ 
& + \tfrac{1}{4}\left|v\left(x^{k_r}\right)\right|^\top\left({\overline{\nabla_{gH}^2 G_i\left(x^{k_r}\right)}-\underline{\nabla_{gH}^2 G_i\left(x^{k_r}\right)}}\right) \left|v\left(x^{k_r}\right)\right|.
	\end{align*}
	Therefore, $\overline{g}_{x^{k_r}}^i\left(v\left(x^{k_r}\right)\right)\to \overline{g}_{\bar{x}}^i\left(\bar{v}\right)$ as $r\to\infty$.
	Consequently, from \eqref{xi-kr-less}, we get 
	\begin{equation}\label{convergence proof eq1}
		\underset{i=1,2,\ldots,m}{\max}\:\overline{g}_{\bar{x}}^i\left(\bar{v}\right)\leq\xi(\bar{x})\leq0.
	\end{equation}
	Let us now consider a fixed but arbitrary positive integer $p$. Since $t_{k_r}\to 0$ as $r\to\infty$ and $\eta\in(0,1)$, we have $t_{k_r}<\eta^p$, which means that the Armijo-like condition \eqref{Armijo condition} at $x^{k_r}$ is not satisfied at $t=\eta^p$. Then, for all $r$, there exists $i=i\left(k_r\right)\in\left\{1,2,\ldots,m\right\}$ such that one of two following inequalities is satisfied:
	\begin{equation*}
		\begin{rcases}
			\begin{aligned}
				&\underline{G}_{i}\left(x^{k_r}+\eta^p v\left(x^{k_r}\right)\right)> \underline{G}_{i}\left(x^{k_r}\right)+\sigma \eta^p \:\xi\left(x^{k_r}\right)\\
				\text{or, }&\overline{G}_{i}\left(x^{k_r}+\eta^p v\left(x^{k_r}\right)\right)> \overline{G}_{i}\left(x^{k_r}\right)+\sigma \eta^p \:\xi\left(x^{k_r}\right).
			\end{aligned} 
		\end{rcases} 
	\end{equation*} 
	Since $\left\{i\left(k_r\right)\right\}_r\subset \left\{1,2,\ldots,m\right\}$, there exists a subsequence $\left\{k_{r_l}\right\}_l$ and an index $i_0$ such that $i_0=i\left(k_{r_l}\right)$ for all $l=1,2,\ldots$, and one of two following inequalities is satisfied: 
		\begin{equation*}
		\begin{rcases}
			\begin{aligned}
				&\underline{G}_{i_0}\left(x^{k_{r_l}}+\eta^p v\left(x^{k_{r_l}}\right)\right)> \underline{G}_{i_0}\left(x^{k_{r_l}}\right)+\sigma \eta^p \:\xi\left(x^{k_{r_l}}\right)\\
				\text{or, }&\overline{G}_{i_0}\left(x^{k_{r_l}}+\eta^p v\left(x^{k_{r_l}}\right)\right)> \overline{G}_{i_0}\left(x^{k_{r_l}}\right)+\sigma \eta^p \:\xi\left(x^{k_{r_l}}\right).
			\end{aligned} 
		\end{rcases} 
	\end{equation*} 
	 Taking $l\to\infty$, we get
	 	\begin{equation*}
	 	\begin{rcases}
	 		\begin{aligned}
	 			&\underline{G}_{i_0}\left(\bar{x}+\eta^p \bar{v}\right)\geq \underline{G}_{i_0}\left(\bar{x}\right)+\sigma \eta^p \:\xi\left(\bar{x}\right)\\
	 			\text{or, }&\overline{G}_{i_0}\left(\bar{x}+\eta^p \bar{v}\right)\geq \overline{G}_{i_0}\left(\bar{x}\right)+\sigma \eta^p \:\xi\left(\bar{x}\right).
	 		\end{aligned}
	 	\end{rcases}
	 \end{equation*}
	 Therefore, we get 
	  \[G_{i_0}\left(\bar{x}+\eta^p\bar{v}\right)\nprec G_{i_0}\left(\bar{x}\right)\oplus \left[\sigma \eta^p,\sigma \eta^p\right]\odot \xi\left(\bar{x}\right),\] which is true for any positive integer $p$. Therefore, by Theorem \ref{steplength theorem}, we get $\nabla_{gH} G_{i_0}\left(\bar{x}\right)^\top\odot \bar{v}\nprec {\bf 0}$. Since $\nabla_{gH}^2G_{i_0}\left(\bar{x}\right)\succ{\bf 0}$, we get $\overline{g}_{\bar{x}}^{i_0}\left(\bar{v}\right)\geq0$. Therefore, we get 
	\begin{equation}\label{convergence proof eq2}
		\underset{i=1,2,\ldots,m}{\max}\:\overline{g}_{\bar{x}}^i\left(\bar{v}\right)\geq0.
	\end{equation}
	Using \eqref{convergence proof eq1} and \eqref{convergence proof eq2}, we get $\xi(\bar{x})=0$. Therefore, by Theorem \ref{descent direction finding theorem}, $\bar{x}$ is a Pareto critical point of the MIOP \eqref{minG(x)}.
	\end{description}
	
	Since for all $k\in{\mathbb{N}}$ and for all $i=1,2,\ldots,m$, $G_i\left(x^{k+1}\right)\preceq G_i\left(x^{k}\right)$, the sequence $\left\{x^k\right\}$ contained in the bounded level set $L_0:=\left\{x\in{\mathbb{R}}^n:G_i(x)\preceq G_i(x^0),i=1,2,\ldots m\right\}$. So, the sequence $\left\{x^k\right\}$ is bounded, and it has at least one accumulation point and this completes the proof.
\end{proof}

\section{Numerical Experiments}\label{Numerical experiments}
In this section, we show the performance of the proposed Algorithm \ref{Algorithm} on some test problems given in Appendix \ref{Test problems}. Algorithm \ref{Algorithm} is executed in MATLAB2023a software. This MATLAB software is installed on a 10th GEN PC with processor Intel(R) Core(TM) i5-1035G1 CPU @ 1.00GHz   1.19 GHz and 8 GB RAM. 

In the MATLAB execution of Algorithm \ref{Algorithm}, we take the following parameter values.
\begin{itemize}
	\item To choose the initial point, we use the lower and the upper bounds of the decision variables for each test problem as given in Appendix \ref{Test problems}. The initial point $x^0$ is chosen randomly from the domain of the test problem using ``rand" function of MATLAB.
	\item For all the test problems, we choose the step length reduction factor $\eta=\tfrac{1}{2}$ and the line search parameter $\sigma=0.001$. To find the step length $t_k$ for each $k$, we use the conditions given in \eqref{Armijo condition}. 
	\item To find $v\left(x^k\right)$ and $\xi\left(x^k\right)$ for each $k$, we solve the subproblem given in \eqref{equivalent constrained problem} using ``fmincon" MATLAB optimization toolbox.
	\item According to Algorithm \ref{Algorithm}, we set the stopping condition by $\xi\left(x^k\right)>-\epsilon$. Here, we choose the tolerance level $\epsilon=10^{-6}$.
\end{itemize}
With these parameter values, we first solve the test problem I-BK1. For this problem, we choose the initial point $x^0=\left(9.9862,-7.4332\right)^\top$. First, we compute $gH$-gradients and $gH$-Hessian  of IVMs $G_1$ and $G_2$. The $gH$-gradients of $G_1$ and $G_2$ are
 \[\nabla_{gH}G_1(x_1,x_2):=\begin{bmatrix}\left[0.2,0.4\right]\odot x_1 \\ \left[0.2,0.6\right]\odot x_2\end{bmatrix} \text{ and }\nabla_{gH}G_2(x_1,x_2):=\begin{bmatrix}\left[0.2,0.6\right]\odot \left(x_1-5\right) \\ \left[0.2,1\right]\odot \left(x_2-5\right)\end{bmatrix}.\] 
 The $gH$-Hessian of $G_1$ and $G_2$ are
 \[\nabla_{gH}^2G_1(x_1,x_2):=\begin{bmatrix}\left[0.2,0.4\right] & [0,0]\\ [0,0] & \left[0.2,0.6\right]\end{bmatrix} \text{ and }\nabla_{gH}^2G_2(x_1,x_2):=\begin{bmatrix}\left[0.2,0.6\right] & [0,0] \\ [0,0] & \left[0.2,1\right]\end{bmatrix}.\]
 Next, we show the iterations as follow:
 \begin{description}
 	\item[{\bf Iteration 1:}] At $x^0=\left(9.9862,-7.4332\right)^\top$, we have 
 	\[\nabla_{gH}G_1\left(x^0\right)=\begin{bmatrix}\left[1.99724,3.99448\right] \\ \left[-4.45992,-1.48664\right]\end{bmatrix}, \nabla_{gH}G_2\left(x^0\right)=\begin{bmatrix}\left[0.99724,2.99172\right] \\ \left[-12.4332,-2.48664\right]\end{bmatrix},\] 
 	\[\nabla_{gH}^2G_1\left(x^0\right)=\begin{bmatrix}\left[0.2,0.4\right] & [0,0]\\ [0,0] & \left[0.2,0.6\right]\end{bmatrix}, \text{ and }\nabla_{gH}^2G_2\left(x^0\right)=\begin{bmatrix}\left[0.2,0.6\right] & [0,0] \\ [0,0] & \left[0.2,1\right]\end{bmatrix}.\]
 	To find $v\left(x^0\right)$ and $\xi\left(x^0\right)$, we solve the following subproblem:
 	\begin{equation}\label{subproblem at iteration1}
 		\begin{rcases}
 			\begin{aligned}
 				\min_{\substack{{u,v\in{\mathbb{R}}^2}\\ \tau\in{\mathbb{R}}}}&\tau\\
 				\text{subject to } &2.99586v_1-2.97328v_2+0.99862u_1+1.48664u_2+0.15v_1^2+0.2v_2^2+0.05 u_1^2+0.1u_2^2\leq\tau,\\
 				&1.99448v_1-7.45992v_2+0.99724u_1+4.97328u_2+0.2v_1^2+0.3v_2^2+0.1 u_1^2+0.2u_2^2\leq\tau,\\
 				&-u_1\leq v_1 \leq u_1,~-u_2\leq v_2 \leq u_2.
 			\end{aligned}
 		\end{rcases}
 	\end{equation}
  We get the optimal solution $v\left(x^0\right)$ and the optimal value $\xi\left(x^0\right)$ of the subproblem \eqref{subproblem at iteration1} using ``fmincon" MATLAB optimization toolbox as $v\left(x^0\right)=\left(-1.6621, 2.4866\right)^\top$ and $\xi\left(x^0\right)=-3.920429$. Since $\xi\left(x^0\right)=-3.920429<-\epsilon$, we have to find the next iterative point. In this regard, we find the step length using Armijo-like rule. We see that $t=1$ satisfies the conditions given in \eqref{Armijo condition}. So, we get $t_0=1$. Therefore, the next iterative point is \[x^1=x^0+t_0v\left(x^0\right)=\left(9.9862,-7.4332\right)^\top+\left(-1.6621, 2.4866\right)^\top=\left(8.3241,-4.9466\right)^\top.\]
  \item[{\bf Iteration 2:}] At $x^1=\left(8.3241,-4.9466\right)^\top$, we have 
  \[\nabla_{gH}G_1\left(x^1\right)=\begin{bmatrix}\left[1.66482,3.32964\right] \\ \left[-2.96796,-0.98932\right]\end{bmatrix}, \nabla_{gH}G_2\left(x^1\right)=\begin{bmatrix}\left[0.66482,1.99446\right] \\ \left[-9.9466,-1.98932\right]\end{bmatrix},\]
  \[\nabla_{gH}^2G_1\left(x^1\right)=\begin{bmatrix}\left[0.2,0.4\right] & [0,0]\\ [0,0] & \left[0.2,0.6\right]\end{bmatrix}, \text{ and }\nabla_{gH}^2G_2\left(x^1\right)=\begin{bmatrix}\left[0.2,0.6\right] & [0,0] \\ [0,0] & \left[0.2,1\right]\end{bmatrix}.\]
  To find $v\left(x^1\right)$ and $\xi\left(x^1\right)$, we solve the following subproblem:
  \begin{equation}\label{subproblem at iteration2}
  	\begin{rcases}
  		\begin{aligned}
  			\min_{\substack{{u,v\in{\mathbb{R}}^2}\\ \tau\in{\mathbb{R}}}}&\tau\\
  			\text{subject to } &2.49723v_1-1.97864v_2+0.83241u_1+0.98932u_2+0.15v_1^2+0.2v_2^2+0.05 u_1^2+0.1u_2^2\leq\tau,\\
  			&1.32964v_1-5.96796v_2+0.66482u_1+3.97864u_2+0.2v_1^2+0.3v_2^2+0.1 u_1^2+0.2u_2^2\leq\tau,\\
  			&-u_1\leq v_1 \leq u_1,~-u_2\leq v_2 \leq u_2.
  		\end{aligned}
  	\end{rcases}
  \end{equation}
  We get the optimal solution $v\left(x^1\right)$ and the optimal value $\xi\left(x^1\right)$ of the subproblem \eqref{subproblem at iteration2} using ``fmincon" MATLAB optimization toolbox as $v\left(x^1\right)=\left(-1.1080, 1.9893\right)^\top$ and $\xi\left(x^1\right)=-2.347010$. Since $\xi\left(x^1\right)=-2.347010<-\epsilon$, we have to find the next iterative point. Now, we find the step length using Armijo-like rule. We see that $t=1$ satisfies the conditions given in \eqref{Armijo condition}. So, we get $t_1=1$. Therefore, the next iterative point is \[x^2=x^1+t_1v\left(x^1\right)=\left(8.3241,-4.9466\right)^\top+\left(-1.1080, 1.9893\right)^\top=\left(7.2161,-2.9573\right)^\top.\]
\end{description}
   Continue the iterations until the stopping condition ($\xi\left(x^k\right)>-\epsilon$) met, and for all such iterations, we compute $x^k$, $G\left(x^k\right)$, and $\xi\left(x^k\right)$ which are given in Table \ref{table.1}. It is observed that the stopping condition met after $12$ iterations.
\begin{table}[htbp]
	\caption{ Output of Algorithm \ref{Algorithm} for I-BK1 problem }
	\label{table.1}
	\centering
	\begin{tabular}{c c c c }
		\toprule
		$k$ & $\left(x^k\right)^\top$ & $G\left(x^k\right)^\top$ & $\xi\left(x^k\right)$\\ [0.5ex]
		\midrule
		1 & $\left(8.324133, -4.946560\right)$ & $\left(\left[9.375965, 21.198776\right],\left[10.998392, 52.781987\right]\right)$ &  $-2.347010$\\
		2 & $\left(7.216089, -2.957248\right)$ & $\left(\left[6.081725, 13.037982\right],\left[6.822885, 33.132213\right]\right)$ &  $-1.412520
		$ \\
		3 & $\left(6.242678, -1.410758\right)$ & $\left(\left[4.096127, 8.391277\right],\left[4.264207, 21.012185\right]\right)$  & $-0.815852$\\
		4 &  $\left(5.415449, -0.241812\right)$ & $\left(\left[2.938556, 5.882960\right],\left[2.764919, 13.790076\right]\right)$ &  $-0.453591$\\
		5 & $\left(4.745411, 0.622544\right)$ & $\left(\left[2.290648, 4.620053\right],\left[1.922694, 9.600505\right]\right)$ &  $-0.162046$\\
		6 &  $\left(4.259970, 1.107272\right)$ & $\left(\left[1.937339, 3.997284\right],\left[1.570098, 7.740959\right]\right)$ & $-0.027530$\\
		7 & $\left(4.049701, 1.304596\right)$ & $\left(\left[1.810204, 3.790606\right],\left[1.455908, 7.098927\right]\right)$ & $-0.004217$\\
		8 &  $\left(3.966205, 1.381552\right)$ & $\left(\left[1.763947, 3.718762\right],\left[1.416190, 6.867201\right]\right)$ & $-0.000626$\\
		9 &  $\left(3.933885, 1.411157\right)$ & $\left(\left[1.746682, 3.692499\right],\left[1.401640, 6.780878\right]\right)$ & $-0.000092$\\
		10 &  $\left(3.921485, 1.422489\right)$ & $\left(\left[1.740152, 3.682651\right],\left[1.396178, 6.748249\right]\right)$ & $-0.000013$\\
		11 &  $\left(3.916742, 1.426820\right)$ & $\left(\left[1.737668, 3.678918\right],\left[1.394106, 6.735842\right]\right)$ & $-0.000002$\\
		 12 &   $\left(3.914930, 1.428474\right)$ & $\left(\left[1.736722, 3.677497\right],\left[1.393317, 6.731112\right]\right)$ & $-2.8154e-07$\\[1ex]
		\bottomrule

	\end{tabular}
	
\end{table}\\
For graphical visualization, we depict the value of $G\left(x^k\right)$ and the value of $G\left(x^\star\right)$ in the objective feasible region generated by Algorithm \ref{Algorithm} with the initial point $x^0=\left(9.9862,-7.4332\right)^\top$ for the test problem I-BK1 in left and right hand side of Figure \ref{figure:test}, respectively. Note that each rectangle filled with pistachio color represents $G\left(x^k\right)$, the rectangle filled with golden yellow color represents $G\left(x^\star\right)$. In addition, black bullet points are the center of the rectangles $G\left(x^k\right)$ and the blue bullet point is the center of the rectangle $G\left(x^\star\right)$. The path with magenta color starting from black bullet point and ending at a blue bullet point is the trajectory by the center of the sequence $\left\{G\left(x^k\right)\right\}$ generated by Algorithm \ref{Algorithm}. The light blue shaded region is the objective feasible region. We take 5000 random points from the domain of the test problem. For each point $x$, we get a rectangle $G(x)$. Clearly, union of all such rectangles is the objective feasible region, i.e., $\text{objective feasible region}:=\underset{lb\leq x \leq ub}{\bigcup}G(x)$. In the objective feasible region, the black bullet point is the center of the rectangle $G$ at $x^0=\left(9.9862,-7.4332\right)^\top$ and the blue bullet point is the center of the rectangle $G$ at $x^\star=\left(3.914930, 1.428474\right)^\top$.
  
\begin{figure}[htbp]
	\begin{subfigure}[t]{0.45\textwidth}
		\includegraphics[width=\linewidth]{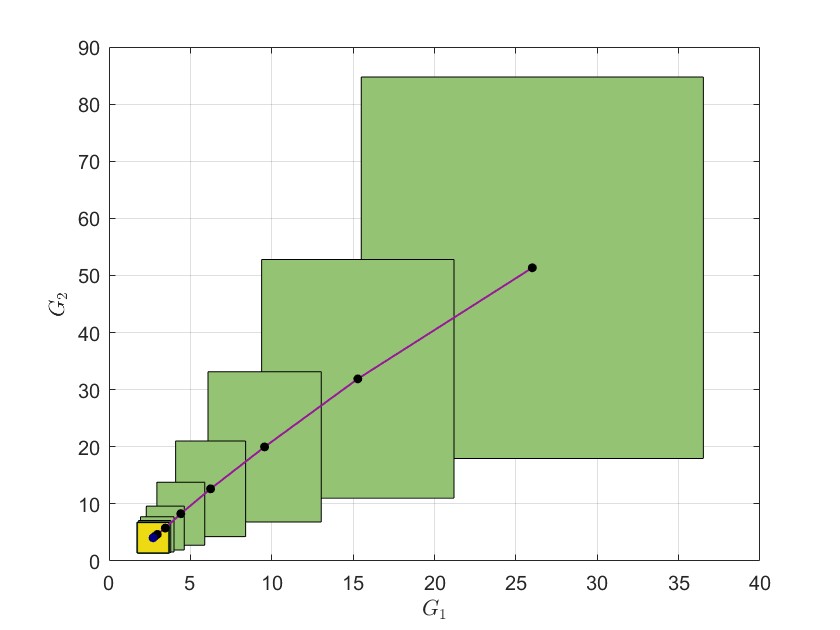}
		\caption{The value of $G\left(x^k\right)$ generated by Algorithm \ref{Algorithm} with the initial point $x^0=\left(9.9862,-7.4332\right)^\top$ for the test problem I-BK1}
		\label{fig:I-BK1-it}
	\end{subfigure}\hfill
	\begin{subfigure}[t]{0.45\textwidth}
		\includegraphics[width=\linewidth]{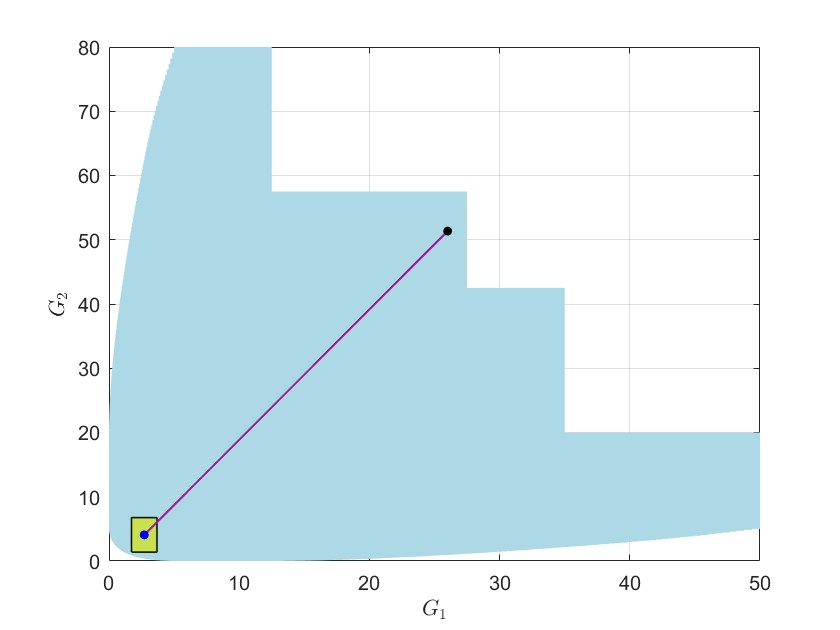}
		\caption{The value of $G\left(x^\star\right)$ in the objective feasible region for the test problem I-BK1}
		\label{fig:I-BK1}
	\end{subfigure}
	\caption{I-BK1.}
	\label{figure:test}
	\end{figure}
	We now incorporate a comparison with the existing methods. In the existing literature, there are Newton \cite{upadhayay2024newton} and quasi-Newton \cite{upadhayay2024quasi} methods for MIOPs. However, in these two articles, authors transformed an MIOP into a real-valued multiobjective optimization problem considering each objective function as a sum of the lower and the upper boundary functions of each IVM and further applied the Newton and quasi-Newton methods for real-valued multiobjective optimization, which essentially captures very tiny part of the entire set of efficient solutions. One can trivially find that almost entire part of the efficient solutions cannot be captured by the methods in \cite{upadhayay2024newton,upadhayay2024quasi}. So, for comparison, we adopt weighted sum method \cite{miettinen1999nonlinear,ehrgott2005multicriteria} for an MIOP to convert it in a single-objective interval optimization problem, and further find the efficient solution by the scalarization method in Bhurjee and Panda \cite{bhurjee2012efficient}. Let $\alpha\in[0,1]$. Using weighted sum method, the MIOP I-BK1 becomes in a single-objective interval optimization problem as \begin{equation}\label{verification1}
		\underset{x_1,x_2}{\min}\left[(1-\alpha)\odot\left([0.1,0.2]\odot x_1^2\oplus[0.1,0.3]\odot x_2^2\right)\oplus \alpha\odot\left(\left[0.1,0.3\right]\odot \left(x_1-5\right)^2\oplus\left[0.1,0.5\right]\odot \left(x_2-5\right)^2\right)\right].
	\end{equation}
	Using scalarization method with the weight function $w(t_1,t_2,t_3,t_4):=t_2+t_3$, the single-objective interval optimization problem \eqref{verification1} becomes as 
	\begin{equation}\label{verification2}
		\underset{x_1,x_2}{\min}\left[(1-\alpha)\left(0.15 x_1^2+0.21667 x_2^2\right)+ \alpha\left(0.21667 \left(x_1-5\right)^2+0.3 \left(x_2-5\right)^2\right)\right],
	\end{equation}
	which is a real-valued single-objective optimization problem. The solution of the problem \eqref{verification2} is 
	\begin{equation}\label{verification solution}
		x_1:=\tfrac{2.1667\alpha}{0.3+0.13334\alpha}\text{ and }x_2:=\tfrac{3\alpha}{0.43334+0.16666\alpha}, \quad \alpha \in[0,1].
	\end{equation}
	For different values of $\alpha\in[0,1]$, the solutions of the problem \eqref{verification2} and the objective values of the IVMs at the corresponding points are given in Table \ref{table.2}
	\begin{table}[htbp]
		\caption{ Solution by weighted sum method for I-BK1 problem }
		\label{table.2}
		\centering
		\begin{tabular}{c c c}
			\toprule
			$\alpha$ & $x^\top$ & $G\left(x\right)^\top$ \\ [0.5ex]
			\midrule
			0 & $\left(0.000000, 0.000000\right)$ & $\left(\left[0.000000,0.000000\right],\left[5.000000,20.000000\right]\right)$ \\
			0.1 & $\left(0.691498, 0.666658\right)$ & $\left(\left[0.092260, 0.228963\right],\left[3.734104, 14.957883\right]\right)$ \\
			0.2 & $\left(1.326546, 1.285699\right)$ & $\left(\left[0.341274, 0.847851\right],\left[2.729029, 10.946295\right]\right)$ \\
			0.3 &  $\left(1.911783, 1.862051\right)$ & $\left(\left[0.712214, 1.771153\right],\left[1.938380, 7.784487\right]\right)$  \\
			0.4 & $\left(2.452849, 2.399981\right)$ & $\left(\left[1.177638, 2.931263\right],\left[1.324898, 5.326893\right]\right)$  \\
			0.5 &  $\left(2.954564, 2.903207\right)$ & $\left(\left[1.715805, 4.274473\right],\left[0.858035, 3.453412\right]\right)$ \\
			0.6 & $\left(3.421069, 3.374983\right)$ & $\left(\left[2.309422, 5.757895\right],\left[0.513370, 2.068246\right]\right)$ \\
			0.7 &  $\left(3.855945, 3.818168\right)$ & $\left(\left[2.944672, 7.347185\right],\left[0.270559, 1.091021\right]\right)$\\
			0.8 &  $\left(4.262305, 4.628566\right)$ & $\left(\left[3.959086, 10.060535\right],\left[0.068215, 0.232240\right]\right)$ \\
			0.9 &  $\left(4.642862, 4.628566\right)$ & $\left(\left[4.297979, 10.738321\right],\left[0.026550, 0.107246\right]\right)$ \\
			1 &   $\left(5.000069, 5.000000\right)$ & $\left(\left[5.000069, 12.500138\right],\left[0.000000, 0.000000\right]\right)$ \\[1ex]
			\bottomrule

		\end{tabular}
		
	\end{table}\\ 
	In Table \ref{table.2}, each point $x$ is a Pareto optimal point generated by the weighted sum scalarization method. However, we see that no point $x$ given in Table \ref{table.2} dominates the point $x^\star=\left(3.914930, 1.428474\right)^\top$ generated by Algorithm \ref{Algorithm} (see Table \ref{table.1}). As the points $x$ generated by the weighted sum scalarization method (see Table \ref{table.2}) and the point $x^\star=\left(3.914930, 1.428474\right)^\top$ generated by Algorithm \ref{Algorithm} (see Table \ref{table.1}) are nondominated each others, the point $x^\star=\left(3.914930, 1.428474\right)^\top$ is also a Pareto optimal point for the MIOP I-BK1. Consequently, by Lemma \ref{interrelation lemma of Pareto optimal and critical}, the point $x^\star=\left(3.914930, 1.428474\right)^\top$ is a Pareto critical point. Hence, the verification of the proposed algorithm is done.
	
	Note that there does not exist any $\alpha\in[0,1]$ such that the point $x^\star=\left(3.914930, 1.428474\right)^\top$ satisfies \eqref{verification solution}. We see that the point $x^\star=\left(3.914930, 1.428474\right)^\top$ is a Pareto optimal point for the MIOP I-BK1, however, it cannot be obtained by the weighted sum scalarization method. Although the MIOP I-BK1 is a convex problem, the weighted sum method fails to find the Pareto optimal point $x^\star=\left(3.914930, 1.428474\right)^\top$. This is a drawback of the weighted sum method. The major drawback of the weighted sum method is its dependence on outside the problem data, which is a burden task to the decision-maker.
	
	\bigskip
	
	To show the performance of Algorithm \ref{Algorithm}, we compute iteration numbers and CPU time in seconds for each test problems, which are given in Table \ref{Performance table}. Taking 100 randomly chosen initial points, we compute Min, Max, Mean, Median, Mode, and Std. Dev. (Standard Deviation) for both of iteration numbers and CPU time  for each test problems. Further, we depict the performance profile in Figure \ref{figure:performance profile} from the perspective of Dolan-Mor{\'e} \cite{dolan2002benchmarking} performance profile to compare the performance of Algorithm \ref{Algorithm} with the algorithm of the steepest descent method \cite{mondal2025steepest}. Let $\mathcal{S}$ and $\mathcal{P}$ be the set of solvers and the set of problems, respectively. In addition, we assume that $N_s$ and $N_p$ are the number of solvers and the number of problems, respectively. We are interested in using average iteration numbers and average CPU time as performance measures. For each problem $p$ and solver $s$, we define \[I_{p,s}:=\text{ average number of iterations required to solve problem }p \text{ by solver }s\] and \[T_{p,s}:=\text{ average computing time required to solve problem }p \text{ by solver }s.\]
	We compare the performance on problem $p$ by solver $s$ with the best performance by any solver on this problem, i.e., we define the performance ratio for average iteration numbers and average CPU time by 
	\[R_{p,s}^I:=\tfrac{I_{p,s}}{\min\left\{I_{p,s}:s\in{\mathcal{S}}\right\}} \text{ and }R_{p,s}^T:=\tfrac{T_{p,s}}{\min\left\{T_{p,s}:s\in{\mathcal{S}}\right\}}.\] 
	The performance profile $\rho_I:{\mathbb{R}}\to[0,1]$ measured by average iteration numbers and the performance profile $\rho_T:{\mathbb{R}}\to[0,1]$ measured by average CPU time are defined by \[\rho_I\left(\zeta\right):=\tfrac{1}{N_p}\text{ size }\left\{p\in{\mathcal{P}}:R_{p,s}^I\leq \zeta\right\} \text{ and } \rho_T\left(\zeta\right):=\tfrac{1}{N_p}\text{ size }\left\{p\in{\mathcal{P}}:R_{p,s}^T\leq \zeta\right\}\text{ for all } \zeta\in {\mathbb{R}}.\]
	Note that $\rho_I(\zeta)$ and $\rho_T(\zeta)$ are the probabilities for solver $s\in {\mathcal{S}}$ that the performance ratios $R_{p,s}^I$ and $R_{p,s}^T$ are within a factor $\zeta\in{\mathbb{R}}$ of the best possible ratios, respectively. The functions $\rho_I$ and $\rho_T$ are the cumulative distribution functions for the performance ratios $R_{p,s}^I$ and $R_{p,s}^T$, respectively. From Figure \ref{figure:performance profile}, we see that Algorithm \ref{Algorithm} outperformed the algorithm of the steepest descent method \cite{mondal2025steepest}.
	\\
	\begin{table}[h]
		\centering
		\caption{Performance of Algorithm \ref{Algorithm}}
		\begin{tabular}{ccc}
			\toprule
			Problem & Iterations & CPU time \\
			name & (Min, Max, Mean, Median, Mode, Std. Dev.) & (Min, Max, Mean, Median, Mode, Std. Dev.) \\
			\midrule
			I-BK1 & (0, 29, 12.6800, 10.0000, 10, 6.8267) & (0.0133, 1.1959, 0.2485, 0.1975, 0.0133, 0.1802) \\
			I-VU2 & (0, 6, 3.8000, 4.0000, 4, 1.7403)
			 & (0.0114, 0.2136, 0.1028, 0.0755, 0.0114, 0.0651) \\
			I-CH & (0, 7, 2.8600, 3.0000, 1, 1.6082)
			 & (0.0102, 0.1356, 0.0608, 0.0571, 0.0102, 0.0231) \\
			I-FON & (0, 19, 4.1400, 3.0000, 1, 4.3624)
			 & (0.0086, 0.8644, 0.2015, 0.1276, 0.0086, 0.1838) \\
			I-KW2 & (1, 7, 2.6000, 2.0000, 1, 2.0656) & (0.0255, 0.8670, 0.2049, 0.0440, 0.0255, 0.3371)\\
			I-Far1 & (0, 18, 2.0000, 0.0000, 0, 5.6372) & (0.0118, 0.3512, 0.0560, 0.0188, 0.0118, 0.1045) \\
			I-Hil1 & (1, 12, 7.4000, 8.0000, 1, 4.0373)
			 & (0.0595, 1.2087, 0.5173, 0.1482, 0.0595, 0.5614) \\
			I-PNR & (0, 9, 5.9000, 6.0000, 7, 2.5683)
			
			& (0.0083, 0.1473, 0.0787, 0.0800, 0.0083, 0.0341) \\
			I-Deb & (1, 7, 1.7600, 1.0000, 1, 1.7357)
			& (0.0260, 0.9584, 0.2355, 0.0702, 0.0260, 0.2960) \\
			I-SD & (0, 9, 4.3900, 5.0000, 5, 2.2737)
			
			& (0.0114, 0.1534, 0.0833, 0.0833, 0.0114, 0.0361) \\
			I-IKK1 & (0, 189, 6.8700, 3.0000, 3, 21.6401)

			& (0.0087, 2.4322, 0.1399, 0.0945, 0.0087, 0.3002) \\
			I-VFM1 & (0, 18, 10.3800, 11.0000, 0, 5.0147)
			& (0.0084, 0.3866, 0.1639, 0.1588, 0.0084, 0.0814) \\
			I-MHHM2 & (2, 9, 6.6000, 7.0000, 7, 1.5308)
			
			& (0.0414, 0.2065, 0.1071, 0.1040, 0.0414, 0.0268) \\
			I-Viennet & (0, 11, 0.5600, 0.0000, 0, 1.7975)
			& (0.0089, 0.1732, 0.0267, 0.0186, 0.0089, 0.0279) \\
			I-AP1 & (0, 82, 3.8700, 0.0000, 0, 13.2327)
			& (0.0065, 23.9864, 0.9906, 0.0194, 0.0065, 3.6356) \\
			I-MOP7 & (12, 36, 21.1700, 19.0000, 16, 6.4417)
			& (0.1576, 1.1287, 0.4172, 0.3361, 0.1576, 0.1791) \\
			I-VFM2 & (0, 24, 9.0900, 9.0000, 9, 3.9083)
			& (0.0161, 0.4190, 0.1590, 0.1457, 0.0161, 0.0736)\\
			I-TR1 & (10, 15, 12.3333, 12.0000, 10, 2.5166)
			& (0.1283, 0.2156, 0.1654, 0.1589, 0.1283, 0.0452) \\
			I-AP4 & (0, 58, 3.4700, 0.0000, 0, 11.0201)
			
			& (0.0055, 18.2562, 0.6336, 0.0155, 0.0055, 2.4032) \\
			I-Comet & (0, 136, 6.0600, 3.0000, 3, 14.3075)
			& (0.0185, 2.6585, 0.1384, 0.0798, 0.0185, 0.2790) \\
			\bottomrule
		\end{tabular}
		\label{Performance table}
	\end{table}\\
	\begin{figure}[htbp]
		\begin{subfigure}[t]{0.45\textwidth}
			\includegraphics[width=\linewidth]{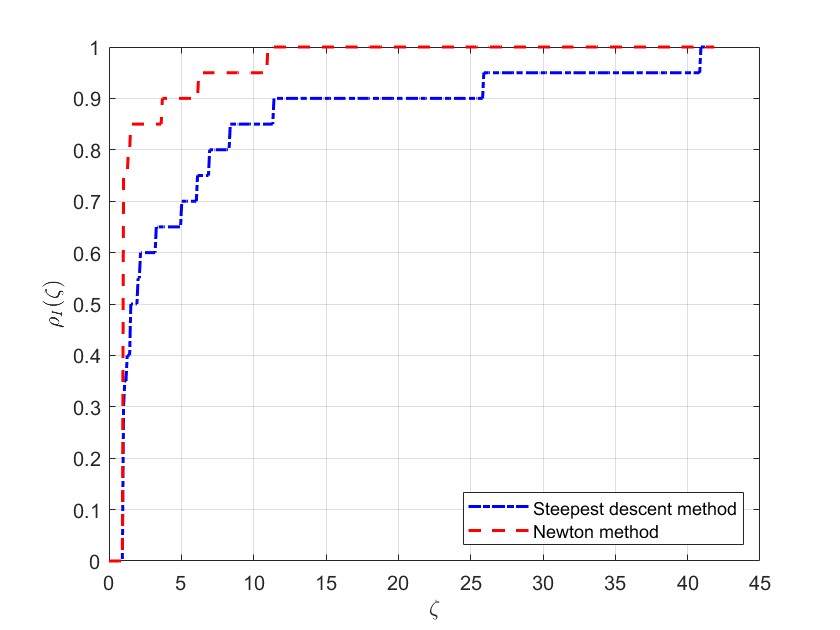}
			\caption{Performance profile of the steepest descent method and the Newton method measured by average iteration numbers}
			\label{iteration}
		\end{subfigure}\hfill
		\begin{subfigure}[t]{0.45\textwidth}
			\includegraphics[width=\linewidth]{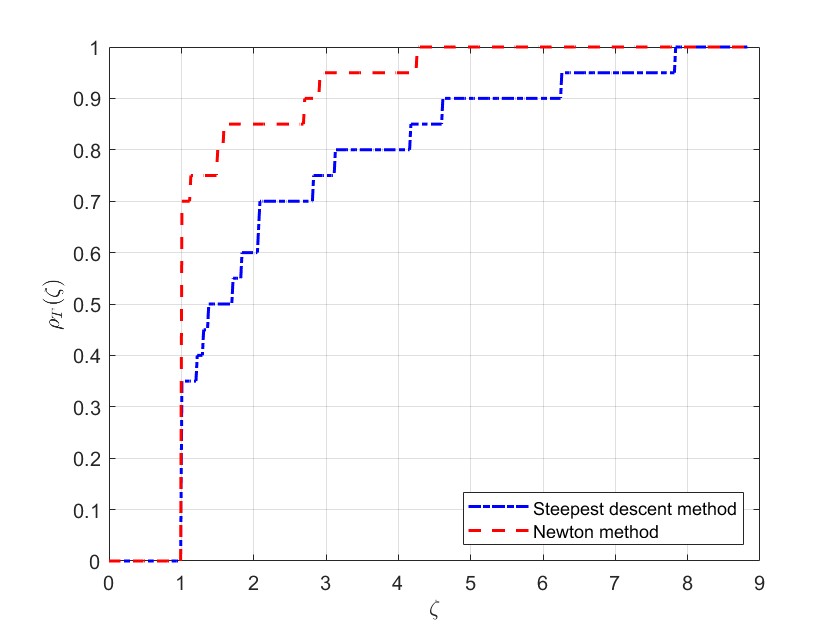}
			\caption{Performance profile of the steepest descent method and the Newton method measured by average CPU time}
			\label{CPU time}
		\end{subfigure}
		\caption{Performance profile.}
		\label{figure:performance profile}
	\end{figure}
	 We depict the objective feasible space for each biobjective test problems given in Appendix \ref{Test problems} and further, for five randomly chosen initial points, we depict $G\left(x^\star\right)$ in the objective feasible space for each biobjective test problems in Figure \ref{figure:biobjective}. Note that the black bullet point is the center of the rectangle $G\left(x^0\right)$ and the blue bullet point is the center of the rectangle $G\left(x^\star\right)$, and we join them by magenta color.
	 
	 \bigskip
	 \noindent
	 For triobjective, it is very difficult to visualize the objective feasible space. So, for a randomly chosen initial point $x^0$, we depict $G\left(x^0\right)$ and $G\left(x^\star\right)$ for each triobjective test problems in Figure \ref{figure:triobjective}. Note that the cube filled with pistachio color represents $G\left(x^0\right)$ and the cube filled with golden yellow color represents $G\left(x^\star\right)$. 
	 The path with magenta color starting from black bullet point and ending at a blue bullet point is the trajectory by the center of the sequence $\left\{G\left(x^k\right)\right\}$ generated by Algorithm \ref{Algorithm}.
\begin{figure}[htbp]
	\begin{subfigure}[t]{0.33\textwidth}
		\includegraphics[width=\linewidth]{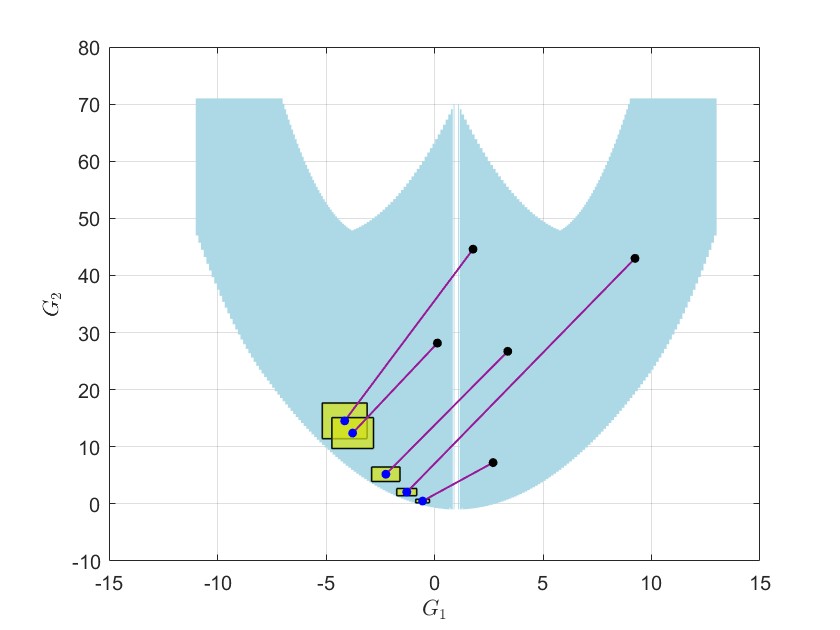}
		\caption{I-VU2}
		\label{fig:I-VU2}
	\end{subfigure}\hfill
	\begin{subfigure}[t]{0.33\textwidth}
		\includegraphics[width=\linewidth]{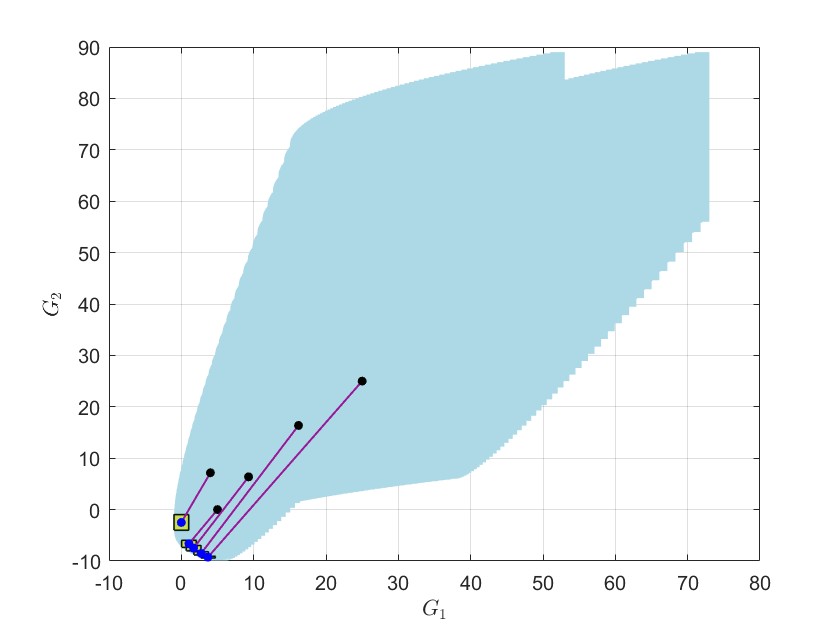}
		\caption{I-CH}
		\label{fig:I-CH}
	\end{subfigure}\hfill
	\begin{subfigure}[t]{0.33\textwidth}
		\includegraphics[width=\linewidth]{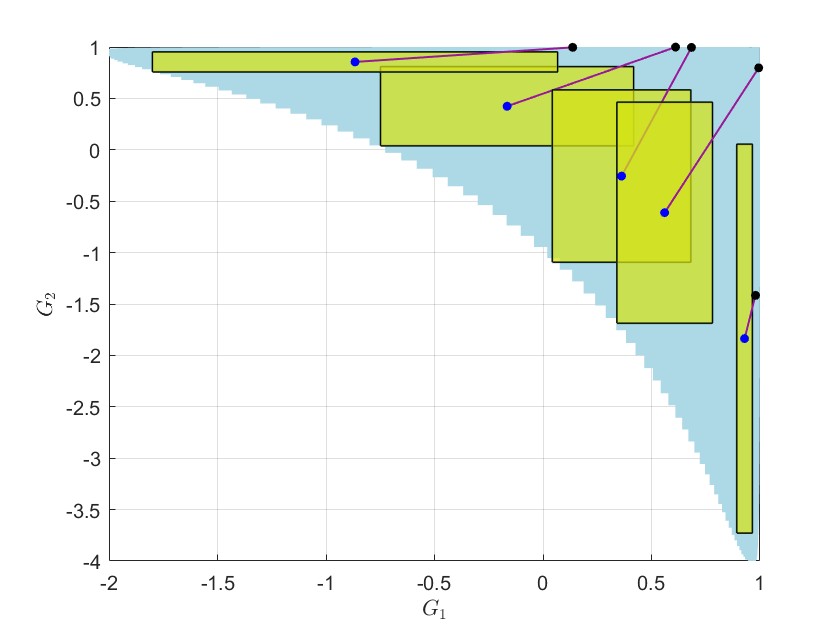}
		\caption{I-FON}
		\label{fig:I-FON}
	\end{subfigure}
	
	\begin{subfigure}[t]{0.33\textwidth}
		\includegraphics[width=\linewidth]{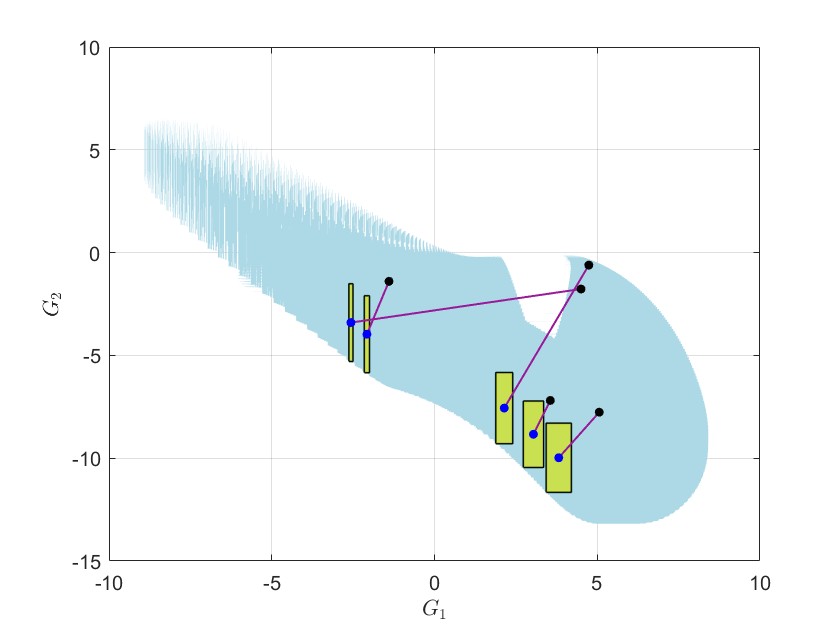}
		\caption{I-KW2}
		\label{fig:I-KW2}
	\end{subfigure}\hfill
	\begin{subfigure}[t]{0.33\textwidth}
		\includegraphics[width=\linewidth]{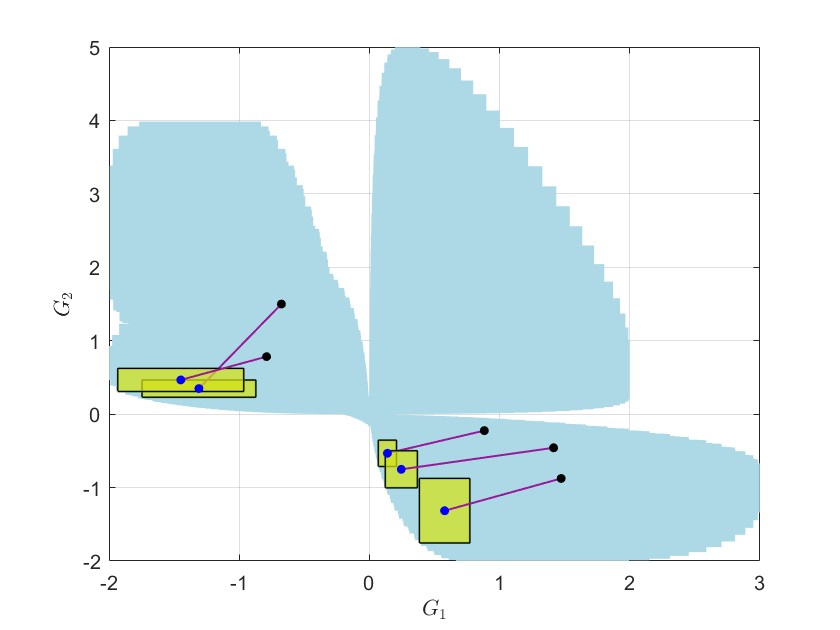}
		\caption{I-Far1}
		\label{fig:I-Far1}
	\end{subfigure}\hfill
	\begin{subfigure}[t]{0.33\textwidth}
		\includegraphics[width=\textwidth]{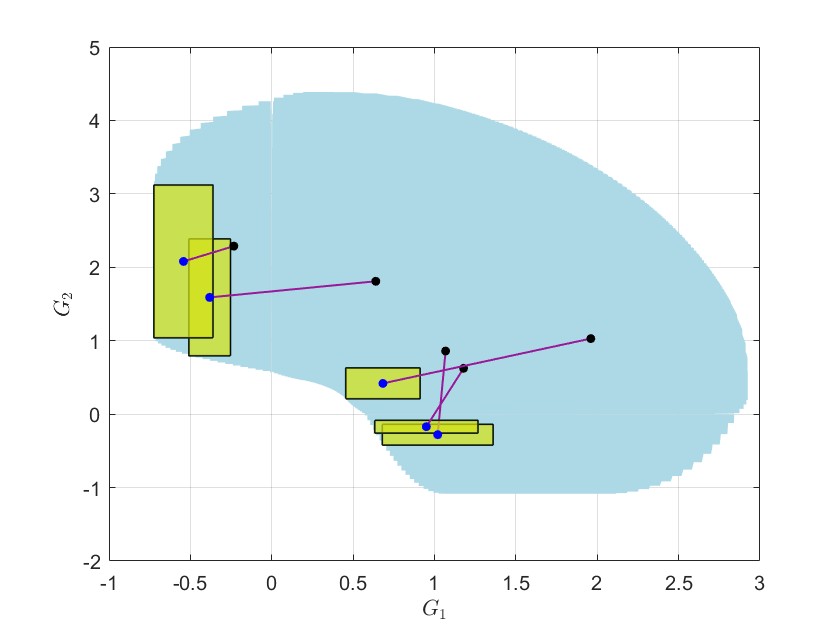}
		\caption{I-Hil1}
		\label{fig:I-Hil1}
	\end{subfigure}
	
	\begin{subfigure}[t]{0.33\textwidth}
		\includegraphics[width=\linewidth]{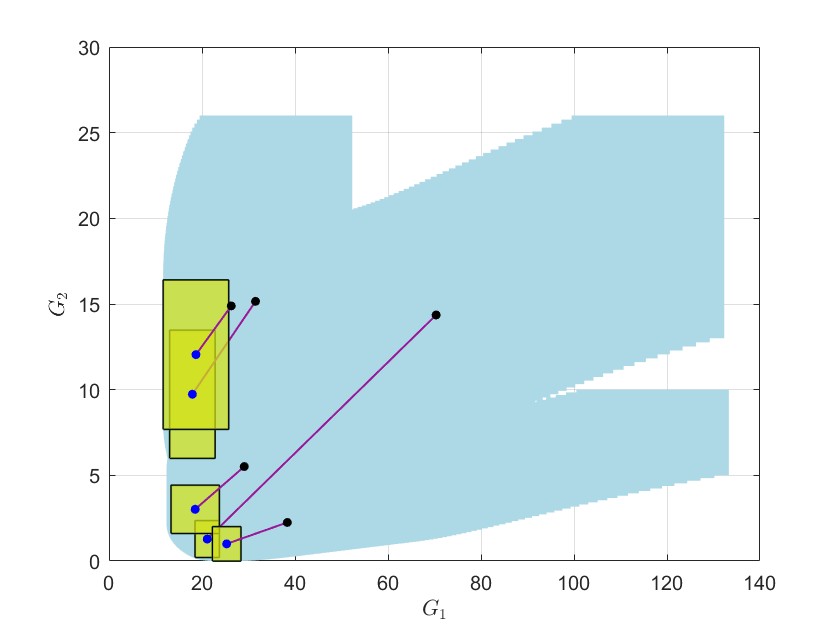}
		\caption{I-PNR}
		\label{fig:I-PNR}
	\end{subfigure}\hfill
	\begin{subfigure}[t]{0.33\textwidth}
		\includegraphics[width=\linewidth]{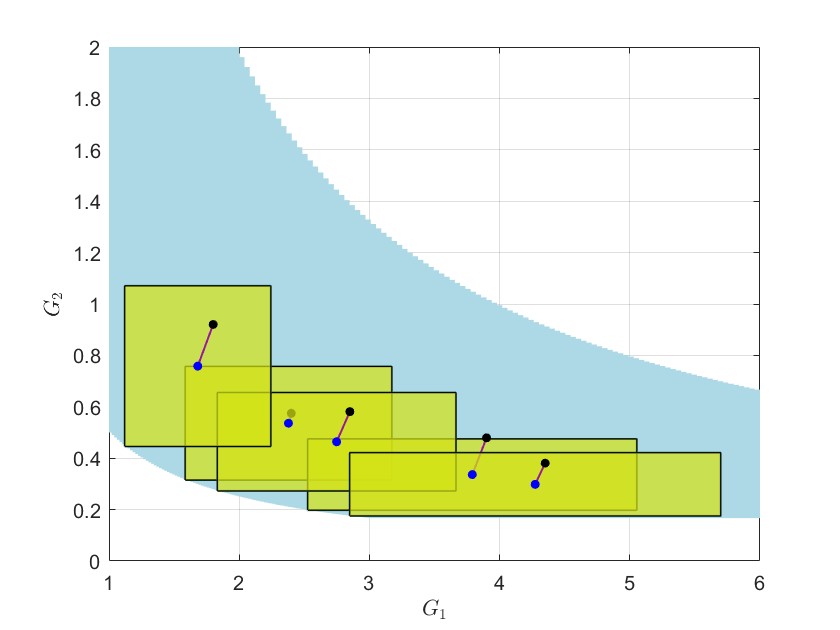}
		\caption{I-Deb}
		\label{fig:I-Deb}
	\end{subfigure}\hfill
	\begin{subfigure}[t]{0.33\textwidth}
		\includegraphics[width=\linewidth]{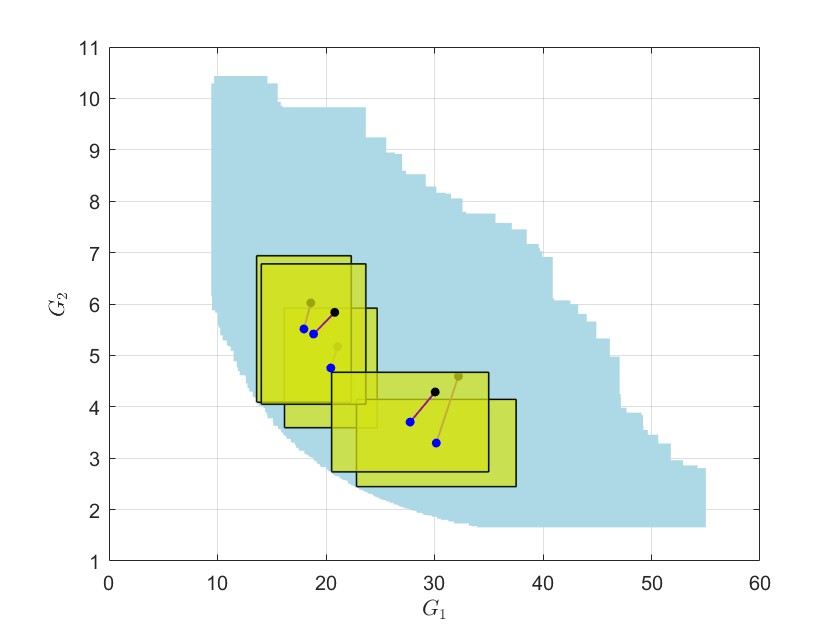}
		\caption{I-SD}
		\label{fig:I-SD}
	\end{subfigure}
	
	\caption{For five randomly chosen initial points, the locations of $G\left(x^\star\right)$ in the objective feasible region of biobjective test problems given in Appendix \ref{Test problems}.}
	\label{figure:biobjective}
\end{figure}
\begin{figure}[htbp]
	\begin{subfigure}[t]{0.33\textwidth}
		\includegraphics[width=\linewidth]{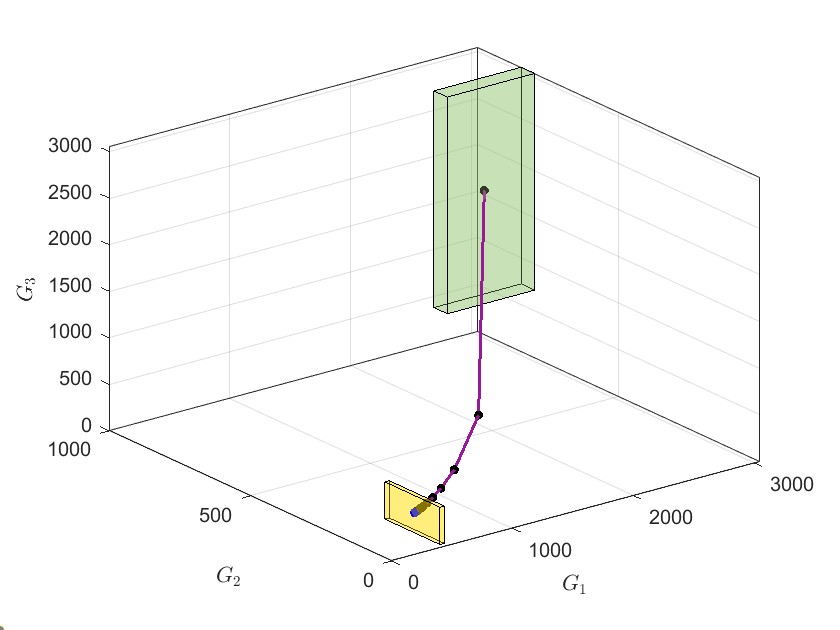}
		\caption{I-IKK1}
		\label{fig:I-IKK1}
	\end{subfigure}\hfill
	\begin{subfigure}[t]{0.33\textwidth}
		\includegraphics[width=\linewidth]{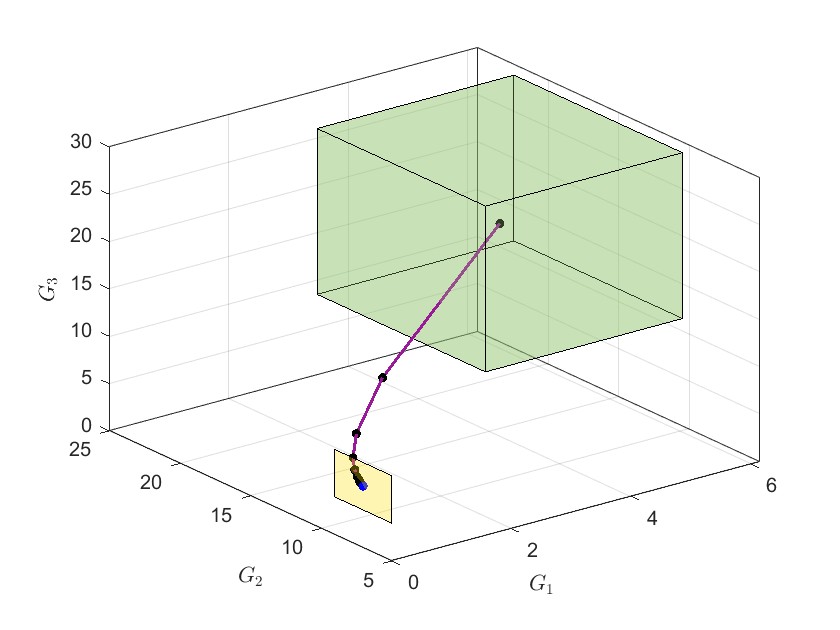}
		\caption{I-VFM1}
		\label{fig:I-VFM1}
	\end{subfigure}\hfill
	\begin{subfigure}[t]{0.33\textwidth}
		\includegraphics[width=\linewidth]{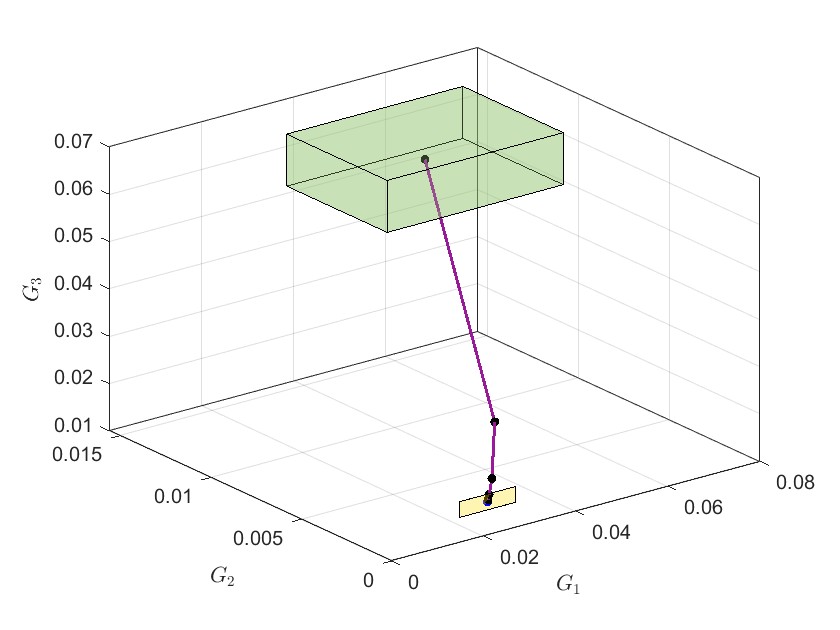}
		\caption{I-MHHM2}
		\label{fig:I-MHHM2}
	\end{subfigure}
	
	\begin{subfigure}[t]{0.33\textwidth}
		\includegraphics[width=\linewidth]{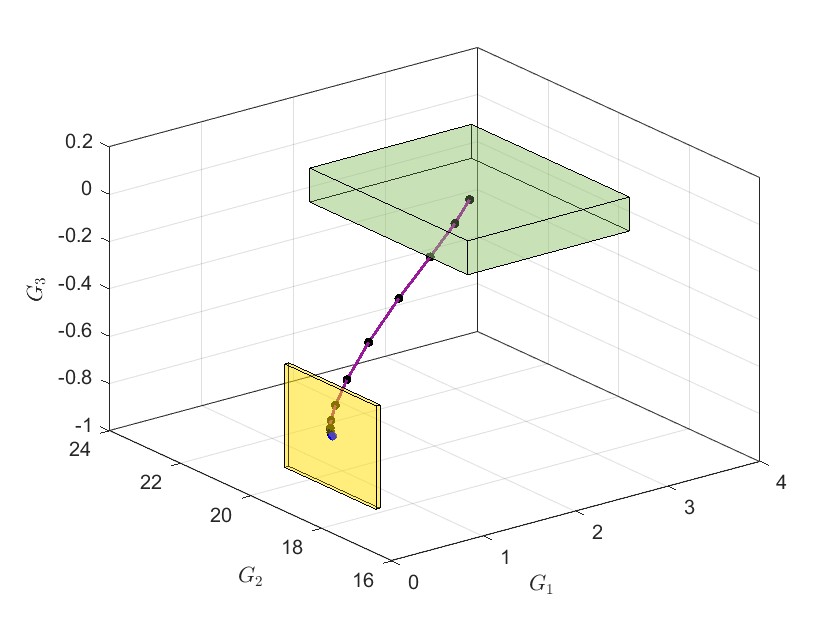}
		\caption{I-Viennet}
		\label{fig:I-Viennet}
	\end{subfigure}\hfill
	\begin{subfigure}[t]{0.33\textwidth}
		\includegraphics[width=\linewidth]{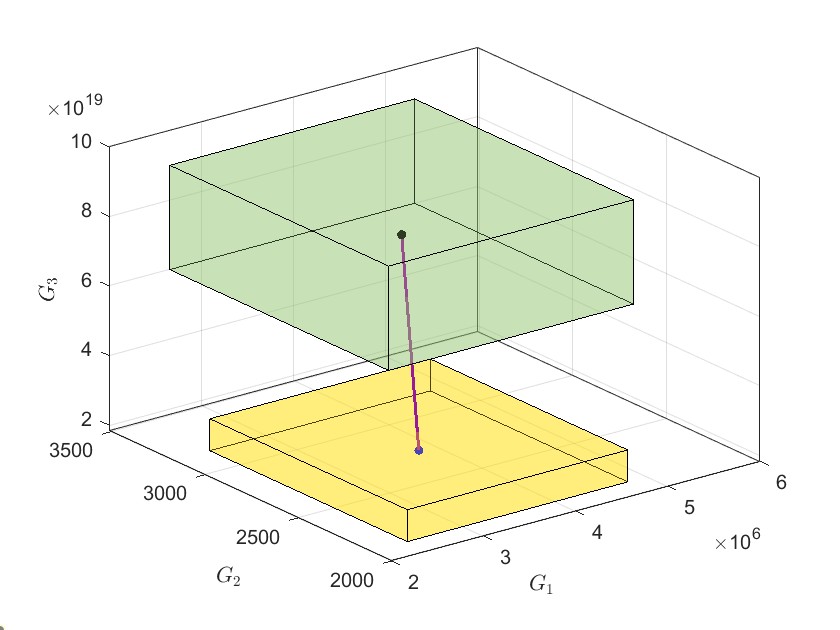}
		\caption{I-AP1}
		\label{fig:I-AP1}
	\end{subfigure}\hfill
	\begin{subfigure}[t]{0.33\textwidth}
		\includegraphics[width=\textwidth]{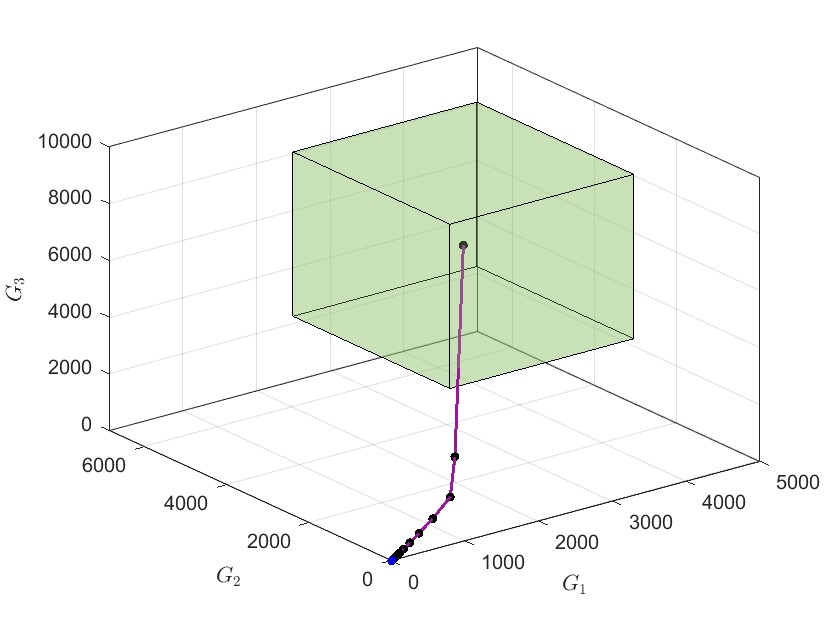}
		\caption{I-MOP7}
		\label{fig:I-MOP7}
	\end{subfigure}
	
	\begin{subfigure}[t]{0.33\textwidth}
		\includegraphics[width=\linewidth]{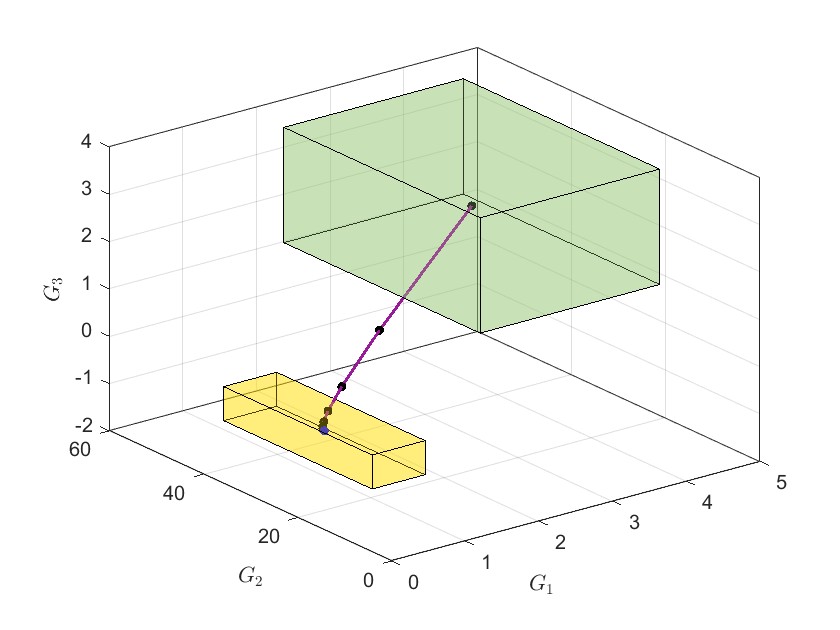}
		\caption{I-VFM2}
		\label{fig:I-VFM2}
	\end{subfigure}\hfill
	\begin{subfigure}[t]{0.33\textwidth}
		\includegraphics[width=\linewidth]{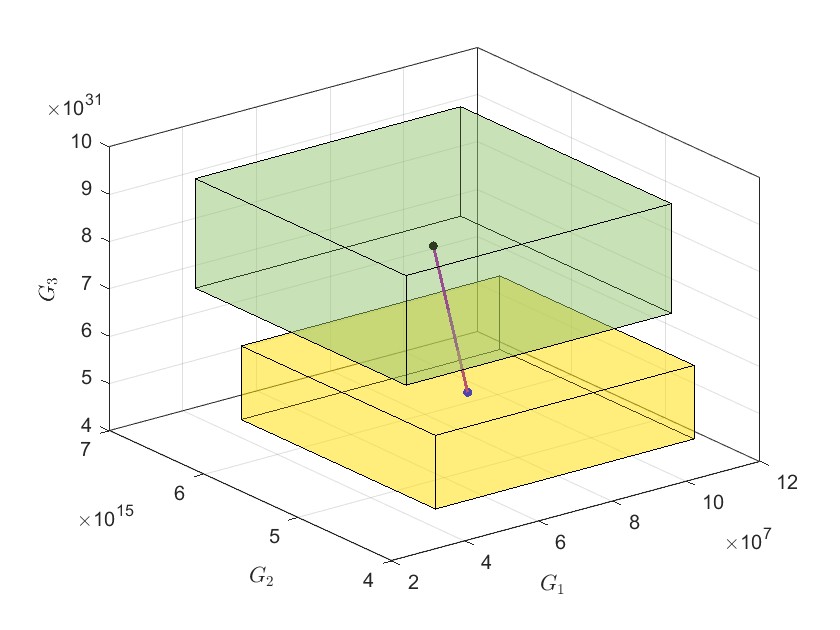}
		\caption{I-AP4}
		\label{fig:I-AP4}
	\end{subfigure}\hfill
	\begin{subfigure}[t]{0.33\textwidth}
		\includegraphics[width=\linewidth]{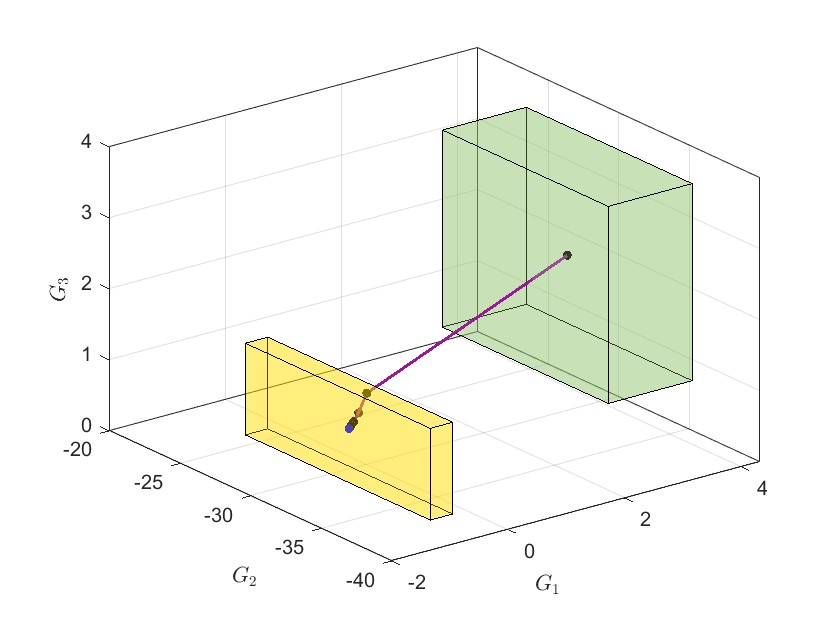}
		\caption{I-Comet}
		\label{fig:I-Comet}
	\end{subfigure}
	
	\caption{For a randomly chosen initial point, the location of $G\left(x^0\right)$ and $G\left(x^\star\right)$ of the triobjective test problems given in Appendix \ref{Test problems}.}
	\label{figure:triobjective}
\end{figure}
\clearpage
\section{Application}\label{Application} In this section, we demonstrate the effectiveness of Algorithm \ref{Algorithm} by applying it to a portfolio optimization problem. Portfolio selection is a fundamental problem in finance, and numerous approaches have been developed to tackle it, including the well-known Markowitz model. To validate the efficiency of our proposed algorithm, we consider a portfolio optimization problem from \cite{upadhayay2024newton}. Consider a portfolio optimization problem that involves two risky assets. Let $r_1$ and $r_2$ be the rate of returns on $x_1$ and $x_2$ portions of total investment funds, respectively. It is well-known that the sum of two portions of total investment funds is always equal to one, i.e., $x_1+x_2=1$. In real-world scenarios, investment returns are inherently uncertain due to market volatility. To model this, the return vector $r_j$ is characterized by its expected return for each asset, denoted as 
$d:=(d_1,d_2)^\top$, and its covariance matrix $\sigma:=\left[\sigma_{ij}\right]_{1\leq i,j\leq 2}$. Here, $\sigma_{ij}$ represents the covariance between $r_i$ and $r_j$, with $\sigma_{ij}=\sigma_{ji}$ for $i,j\in\left\{1,2\right\}$. The expected return of a portfolio is given by $d_1x_1+d_2x_2$, while the associated risk is quantified by the total covariance, expressed as $\sum_{i=1}^2\sum_{j=1}^2\sigma_{ij} x_i x_j$. Therefore, to maximize the expected return and simultaneously minimize the risk factor of the portfolio, it is necessary to solve the following multiobjective optimization problem
\begin{equation}\label{portfolio1}
	\begin{aligned}
		\begin{rcases}
			&\underset{x_1,x_2}{\min} \left(-\left(d_1x_1+d_2x_2\right),\sigma_{11}x_1^2+2\sigma_{12}x_1x_2+\sigma_{22}x_2^2\right)^\top\\
			&\text{subject to } x_1+x_2=1, x_1,x_2\geq 0.
		\end{rcases}
	\end{aligned}
\end{equation}
 In real-world scenarios, it is very difficult task to specify the coefficients $d_1, d_2, \sigma_{11},\sigma_{12},\text{ and } \sigma_{22}$ due to the presence of various types of uncertainty. However, it is possible to estimate lower and upper bounds for these coefficients in some cases. For instance, we assume $d_1=[2,3], d_2=[4,6], \sigma_{11}=[1,2],\sigma_{12}=[-1,0],\text{ and } \sigma_{22}=[2,3]$. Based on these assumptions, the problem \eqref{portfolio1} becomes interval-valued uncertainty framework as \begin{equation}\label{portfolio2}
 	\begin{aligned}
 		\begin{rcases}
 			&\underset{x_1,x_2}{\min} \left((-1)\odot \left(\left[2,3\right]\odot x_1\oplus \left[4,6\right]\odot x_2\right),\left[1,2\right]\odot x_1^2\oplus2\odot\left[-1,0\right]\odot x_1x_2\oplus\left[2,3\right]\odot x_2^2\right)^\top\\
 			&\text{subject to } x_1+x_2=1, x_1,x_2\geq 0.
 		\end{rcases}
 	\end{aligned}
 \end{equation} 
 Replacing $x_2=1-x_1$ in \eqref{portfolio2} and simplifying, we get the following problem 
 \begin{equation}\label{portfolio3}
 	\begin{aligned}
 		\begin{rcases}
 			&\underset{x_1}{\min} \left(\left[3x_1-6,2x_1-4\right],\left[5x_1^2-6x_1+2,5x_1^2-6x_1+3\right]\right)^\top\\
 			&\text{subject to } x_1\in\left[0,1\right].
 		\end{rcases}
 	\end{aligned}
 \end{equation} 
 Applying Algorithm \ref{Algorithm} to the problem \eqref{portfolio3} with the initial points $x_1=0,0.25,0.5,0.75,\text{ and }1$, we get the Pareto optimal solutions, which are given in Table \ref{portfolio solution table}.
 \begin{table}[htbp]
 	\caption{ Solutions of the problem \eqref{portfolio3}}
 	\label{portfolio solution table}
 	\centering
 	\begin{tabular}{c c c}
 		\toprule
 		Sr. No. & Initial point & Pareto optimal point \\ [0.5ex]
 		\midrule
 		1 & $0$ & $\left(0,1\right)^\top$ \\
 		2 & $0.25$ & $\left(0.25,0.75\right)^\top$ \\
 		3 & $0.5$ & $\left(0.5,0.5\right)^\top$ \\
 		4 &  $0.75$ & $\left(0.6,0.4\right)^\top$  \\
 		5 & $1$ & $\left(0.6,0.4\right)^\top$   \\[1ex]
 		\bottomrule

 	\end{tabular}
 	
 \end{table} 
\section{Conclusion and Future Directions}\label{Conclusion and future directions}
In this paper, we have studied the Newton method for an MIOP. We have proved the result related to the Newton direction at a non-Pareto critical point (Theorem \ref{descent direction finding theorem}). To find the step length, we have used the Armijo-like rule. Further, we have proved the result related to the existence of the step length (Theorem \ref{steplength theorem}). For computational purposes, we have provided the complete algorithm of the Newton method for an MIOP (Algorithm \ref{Algorithm}). We have shown that the iteration scheme of our proposed algorithm is scaling independent of variable (Proposition \ref{scaling of variable}). In convergence analysis, we have proved that every sequence generated by Algorithm \ref{Algorithm} converges to the Pareto critical point of the MIOP \eqref{minG(x)} (Theorem \ref{Convergence theorem}). We have shown the validation and the performance of the proposed Algorithm \ref{Algorithm} through some test problems. Finally, we have applied our proposed algorithm to a portfolio optimization problem.

In \cite{debnath2021charaterizations}, it is shown that a single-objective interval optimization problem is equivalent to a biobjective optimization problem. So, one may think the MIOP \eqref{minG(x)} is equivalent to a conventional real-valued multiobjective optimization problem with $2m$ objective functions. However, this approach is possible if all the IVM $G_i$ is convex, have an explicit form in terms of lower and upper boundary functions $\underline{G}_i$ and $\overline{G}_i$, respectively such that for all $x\in\left\{x\in{\mathbb{R}}^n:lb\leq x\leq ub\right\}$ and for all $i=1,2,\ldots,m$,  \[G_i\left(x\right):=\left[\underline{G}_i\left(x\right),\overline{G}_i\left(x\right)\right], \nabla_{gH}G_i\left(x\right):=\left[\nabla \underline{G}_i\left(x\right),\nabla \overline{G}_i\left(x\right)\right],\text{ and }\nabla_{gH}^2G_i\left(x\right):=\left[\nabla^2 \underline{G}_i\left(x\right),\nabla^2 \overline{G}_i\left(x\right)\right].\] 
Since $\underline{G}_i$ and $\overline{G}_i$ may interchange their position, it is very difficult to express the IVM $G_i$ by $G_i\left(x\right):=\left[\underline{G}_i\left(x\right),\overline{G}_i\left(x\right)\right]$ for all $x\in\left\{x\in{\mathbb{R}}^n:lb\leq x\leq ub\right\}$. In addition, the relations $\nabla_{gH}G_i\left(x\right):=\left[\nabla \underline{G}_i\left(x\right),\nabla \overline{G}_i\left(x\right)\right]$ and $\nabla_{gH}^2G_i\left(x\right):=\left[\nabla^2 \underline{G}_i\left(x\right),\nabla^2 \overline{G}_i\left(x\right)\right]$ are not true in general.

In real-valued multiobjective optimization problems, we find the exact form of the Newton direction for convex or strongly convex multiobjective functions as \[v(x):=-\left(\sum_{i=1}^m\lambda_i\nabla^2G_i(x)\right)^{-1}\sum_{i=1}^m\lambda_i \nabla G_i(x),\text{ where }\sum_{i=1}^m\lambda_i=1.\] However, in MIOPs, we are unable to find the exact form of the Newton direction $v(x)$. Due to this reason, we are unable to derive the rate of convergence of our proposed algorithm. We hope that superlinear and quadratic convergence rate of Algorithm \ref{Algorithm} can be proved under the reasonable assumptions. One may think to prove these rate of convergence results in future.

For further research, we will focus on quasi-Newton, conjugate gradient, and trust region methods for an MIOP.
 
\begin{appendices}

\section{List of Test Problems}\label{Test problems}
In this section, we provide a set of test problems for MIOPs refer to \cite{mondal2025steepest}. For each test problem, we provide the expression of the objective functions and the lower and upper bounds of the variables. Accordingly, the MIOP associated to the test problem is \[\min_{\substack{{lb\leq x\leq ub}\\ x\in U\subseteq{\mathbb{R}}^n}}\left(G_1\left(x\right),G_2\left(x\right),\ldots,G_m\left(x\right)\right)^\top.\]
\begin{description}
	\item[\normalfont{ Problem 1 (I-BK1).}\label{I-BK1}] Here, $m=2$, $n=2$, and 
	\begin{align*}
			&G_1(x_1,x_2):=\left[0.1,0.2\right]\odot x_1^2\oplus\left[0.1,0.3\right]\odot x_2^2,\\
			&G_2(x_1,x_2):=\left[0.1,0.3\right]\odot \left(x_1-5\right)^2\oplus\left[0.1,0.5\right]\odot \left(x_2-5\right)^2,\\
			& lb^\top=\left(-10,-10\right) \text{ and } ub^\top=\left(10,10\right).
	\end{align*}
	\item[{\normalfont Problem 2 (I-VU2).}\label{I-VU2}] Here, $m=2,n=2,$ and 
	\begin{align*}
			&G_1(x_1,x_2):=\left[1,1.5\right]\odot x_1\oplus\left[1,1.5\right]\odot x_2\oplus\left[1,1\right],\\
			&G_2(x_1,x_2):=\left[1,1.5\right]\odot x_1^2\oplus\left[2,3\right]\odot x_2^2\ominus_{gH}\left[1,1\right],\\
			& lb^\top=\left(-4,-4\right) \text{ and } ub^\top=\left(4,4\right).
	\end{align*}
	\item[{\normalfont Problem 3 (I-CH).}\label{I-CH}] Here, $m=2,n=2$, and 
	\begin{align*}
			&G_1(x_1,x_2):=[1,1]\odot\left(\left(x_1-1\right)^2+\left(x_2-2\right)^2\right)\oplus\left[-1,1\right],\\
			&G_2(x_1,x_2):=[2,3]\odot \left(x_1^2-x_2\right)\oplus\left[-2,2\right],\\
		& lb^\top=\left(-5,-4\right) \text{ and } ub^\top=\left(5,4\right).
	\end{align*}
	\item[{\normalfont Problem 4 (I-FON).}\label{I-FON}] Here, $m=2,n=2,$ and 
	\begin{align*}
			&G_1(x_1,x_2):=\left[1,1\right]\ominus_{gH}\left[1,3\right]\odot\exp\left(-\left(x_1-\sqrt{\tfrac{1}{2}}\right)^2-\left(x_2-\sqrt{\tfrac{1}{2}}\right)^2\right),\\
			&G_2(x_1,x_2):=\left[1,1\right]\ominus_{gH}\left[1,5\right]\odot\exp\left(-\left(x_1+\sqrt{\tfrac{1}{2}}\right)^2-\left(x_2+\sqrt{\tfrac{1}{2}}\right)^2\right),\\
	&lb^\top=\left(-2,-2\right) \text{ and } ub^\top=\left(2,2\right).
	\end{align*}
		\item[{\normalfont Problem 5 (I-KW2).}\label{I_KW2}] Here, $m=2,n=2$, and 
	\begin{align*}
			&G_1(x_1,x_2):=\left[-5,-3\right]\odot\left(1-x_1\right)^2\exp\left(-x_1^2-\left(x_2+1\right)^2\right)\oplus\left[10,10\right]\odot\left(\tfrac{1}{5}x_1-x_1^3-x_2^5\right)\exp\left(-x_1^2-x_2^2\right)\\
			&\hspace{2cm}\oplus\left[3,5\right]\odot\exp\left(-\left(x_1+2\right)^2-x_2^2\right)\ominus_{gH}\left[\tfrac{1}{2},\tfrac{1}{2}\right]\left(2x_1+x_2\right),\\
			&G_2(x_1,x_2):=\left[-5,-3\right]\odot\left(1+x_2\right)^2\odot\exp\left(-x_2^2-\left(1-x_2\right)^2\right) \\ 
			&\hspace{2cm}\oplus\left[10,10\right]\odot\left(-\tfrac{1}{5}x_2+x_2^3+x_1^5\right)\exp\left(-x_1^2-x_2^2\right) \oplus\left[3,5\right]\odot \exp\left(-\left(2-x_2\right)^2-x_1^2\right),\\
			&lb^\top=\left(-3,-1\right) \text{ and } ub^\top=\left(0,2\right).
	\end{align*}
	\item[{\normalfont Problem 6 (I-Far1).}\label{I-Far1}] Here, $m=2,n=2,$ and 
	\begin{align*}
			&G_1(x_1,x_2):=\left[-2,-1\right]\odot\exp\left(15\left(-\left(x_1-0.1\right)^2-x_2^2\right)\right)\\
			&\hspace{2cm}\oplus\left[-2,-1\right]\odot\exp\left(20\left(-\left(x_1-0.6\right)^2-\left(x_2-0.6\right)^2\right)\right)\\
			&\hspace{2cm}\oplus \left[1,3\right]\odot\exp\left(20\left(-\left(x_1+0.6\right)^2-\left(x_2-0.6\right)^2\right)\right)\\
			&\hspace{2cm}\oplus \left[1,2\right]\odot\exp\left(20\left(-\left(x_1-0.6\right)^2-\left(x_2+0.6\right)^2\right)\right)\\
			&\hspace{2cm}\oplus \left[1,2\right]\odot\exp\left(20\left(-\left(x_1+0.6\right)^2-\left(x_2+0.6\right)^2\right)\right),\\
			&G_2(x_1,x_2):=\left[2,4\right]\odot\exp\left(20\left(-x_1^2-x_2^2\right)\right)\\
			&\hspace{2cm}\oplus\left[1,2\right]\odot\exp\left(20\left(-\left(x_1-0.4\right)^2-\left(x_2-0.6\right)^2\right)\right)\\
			&\hspace{2cm}\oplus \left[-2,-1\right]\odot\exp\left(20\left(-\left(x_1+0.5\right)^2-\left(x_2-0.7\right)^2\right)\right)\\
			&\hspace{2cm}\oplus \left[-2,-1\right]\odot\exp\left(20\left(-\left(x_1-0.5\right)^2-\left(x_2+0.7\right)^2\right)\right)\\
			&\hspace{2cm}\oplus \left[1,5\right]\odot\exp\left(20\left(-\left(x_1+0.4\right)^2-\left(x_2+0.8\right)^2\right)\right),\\
		&lb^\top=\left(-1,-1\right) \text{ and } ub^\top=\left(1,1\right).
	\end{align*}
	\item[{\normalfont Problem 7 (I-Hil1).}\label{I-Hil1}] Here, $m=2, n=2,$ and 
	\begin{align*}
			&G_1(x_1,x_2):=\left[1,2\right]\odot\left(1+\tfrac{1}{2}\cos\left(2\pi x_1\right)\right)\cos\left(\tfrac{2\pi}{360}\left(45+40\sin\left(2\pi x_1\right)+25\sin\left(2\pi x_2\right)\right)\right),\\
			&G_2(x_1,x_2):=\left[1,3\right]\odot\left(1+\tfrac{1}{2}\cos\left(2\pi x_1\right)\right)\sin\left(\tfrac{2\pi}{360}\left(45+40\sin\left(2\pi x_1\right)+25\sin\left(2\pi x_2\right)\right)\right),\\
			&lb^\top=\left(-1,-1\right) \text{ and } ub^\top=\left(1,1\right).
	\end{align*}
		\item[{\normalfont Problem 8 (I-PNR).}\label{I-PNR}] Here, $m=2,n=2$, and 
	\begin{align*}
			&G_1(x_1,x_2):=\left[1,1.5\right]\odot\left(x_1^4+x_2^4\right)\oplus\left[1,2.6\right]\odot \left(x_1^2+x_2^2\right)\oplus \left[10,10\right]\odot x_1x_2\oplus \left[\tfrac{1}{4},\tfrac{1}{4}\right]\odot x_1\oplus\left[20,24\right],\\
			&G_2(x_1,x_2):=\left[1,2\right]\odot\left(x_1-1\right)^2\oplus \left[1,1.5\right]\odot x_2^2\oplus \left[0,2\right],\\
			&lb^\top=\left(-2,-2\right) \text{ and } ub^\top=\left(2,2\right).
	\end{align*}
	\item[{\normalfont Problem 9 (I-Deb).}\label{I-Deb}] Here, $m=2,n=2$, and
	\begin{align*}
			&G_1(x_1,x_2):=\left[1,2\right]\odot x_1,\\
			&G_2(x_1,x_2):=\tfrac{1}{x_1}\odot\left(\left[2,2\right]\ominus_{gH}\left[1,3\right]\odot \exp\left(-\left(\tfrac{x_2-0.2}{0.004}\right)^2\right)\ominus_{gH}\left[0.8,1.5\right]\odot \exp\left(-\left(\tfrac{x_2-0.6}{0.4}\right)^2\right)\right),\\
		&lb^\top=\left(1,-1\right) \text{ and } ub^\top=\left(3,1\right).
	\end{align*}
	\item[{\normalfont Problem 10 (I-SD).}\label{I-SD}] Here, $m=2,n=4$, and 
	\begin{align*}
			&G_1(x_1,x_2,x_3,x_4):=\left[2,3\right]\odot x_1\oplus \left[\sqrt{2},\sqrt{3}\right]\odot x_2\oplus \left[\sqrt{2},\sqrt{3}\right]\odot x_3\oplus \left[1,3\right]\odot x_4,\\
			&G_2(x_1,x_2,x_3,x_4):=\left[2,3\right]\odot\tfrac{1}{x_1}\oplus\left[2\sqrt{2},3\sqrt{3}\right]\odot\tfrac{1}{x_2}\oplus\left[2\sqrt{2},3\sqrt{3}\right]\odot\tfrac{1}{x_3}\oplus\left[2,3\right]\odot\tfrac{1}{x_4},\\
		&lb^\top=\left(1,\sqrt{2},\sqrt{2},1\right) \text{ and } ub^\top=\left(6,6,6,6\right).
	\end{align*}
	\item[{\normalfont Problem 11 (I-IKK1).}\label{I-IKK1}] Here, $m=3,n=2$, and 
	\begin{align*}
			&G_1(x_1,x_2):=\left[1,1\right]\odot x_1^2\oplus \left[0,1\right]\odot x_2^2,\\
			&G_2(x_1,x_2):=\left[1,1\right]\odot \left(x_1-20\right)^2\oplus \left[0,1\right]\odot \left(x_2-20\right)^2,\\
			&G_3(x_1,x_2):=\left[0,1\right]\odot x_1^2\oplus \left[1,1\right]\odot x_2^2,\\
			&lb^\top=\left(-50,-50\right) \text{ and } ub^\top=\left(50,50\right).
	\end{align*}
	\item[{\normalfont Problem 12 (I-VFM1).}\label{I-VFM1}] Here, $m=3,n=2$, and 
	\begin{align*}
			&G_1(x_1,x_2):=\left[1,2\right]\odot x_1^2\oplus \left[1,3\right]\odot \left(x_2-1\right)^2,\\
			&G_2(x_1,x_2):=\left[1,3\right]\odot x_1^2\oplus \left[1,2\right]\odot \left(x_2+1\right)^2\oplus\left[1,1\right],\\
			&G_3(x_1,x_2):=\left[1,2\right]\odot \left(x_1-1\right)^2\oplus \left[1,5\right]\odot x_2^2\oplus\left[2,2\right],\\
			&lb^\top=\left(-2,-2\right) \text{ and } ub^\top=\left(2,2\right).
	\end{align*}
	\item[{\normalfont Problem 13 (I-MHHM2).}\label{I-MHHM2}] Here, $m=3,n=2$, and 
	\begin{align*}
		&G_1(x_1,x_2):=\left[2,3\right]\odot \left(x_1-0.8\right)^2\oplus \left[1,2\right]\odot \left(x_2-0.6\right)^2,\\
		&G_2(x_1,x_2):=\left[1,2\right]\odot \left(x_1-0.85\right)^2\oplus \left[1,1.5\right]\odot \left(x_2-0.7\right)^2,\\
		&G_3(x_1,x_2):=\left[2,2.5\right]\odot \left(x_1-0.9\right)^2\oplus \left[1,1.2\right]\odot \left(x_2-0.6\right)^2,\\
		&lb^\top=\left(0,0\right) \text{ and } ub^\top=\left(1,1\right).
	\end{align*}
	\item[{\normalfont Problem 14 (I-Viennet).}\label{I-Viennet}] Here, $m=3,n=2$, and 
	\begin{align*}
		&G_1(x_1,x_2):=\left[0.5,1\right]\odot \left(x_1^2+x_2^2\right)\oplus \left[1,2\right]\odot \sin\left(x_1^2+x_2^2\right),\\
		&G_2(x_1,x_2):=\left[\tfrac{1}{8},\tfrac{1}{4}\right]\odot \left(3x_1-2x_2+4\right)^2\oplus\left[\tfrac{1}{27},\tfrac{1}{9}\right]\odot \left(x_1-x_2+1\right)^2\oplus\left[15,16\right],\\
		&G_3(x_1,x_2):=\left[\tfrac{1}{4},\tfrac{1}{2}\right]\odot\tfrac{1}{x_1^2+x_2^2+1}\ominus_{gH}\left[0.9,1.1\right]\odot \exp\left(-x_1^2-x_2^2\right),\\
		&lb^\top=\left(-3,-3\right) \text{ and } ub^\top=\left(3,3\right).
	\end{align*}
	\item[{\normalfont Problem 15 (I-AP1).}\label{I-AP1}] Here, $m=3,n=2$, and 
	\begin{align*}
		&G_1(x_1,x_2):=\left[\tfrac{1}{4},\tfrac{1}{2}\right]\odot \left(\left(x_1-1\right)^4+2\left(x_2-2\right)^4\right),\\
		&G_2(x_1,x_2):=\left[1,2\right]\odot\exp\left(\tfrac{x_1+x_2}{2}\right)\oplus\left[1,1.5\right]\odot\left(x_1^2+x_2^2\right),\\
		&G_3(x_1,x_2):=\left[\tfrac{1}{3},\tfrac{1}{2}\right]\odot\left(\exp\left(-x_1\right)+2\exp\left(-x_2\right)\right), \\
		&lb^\top=\left(-100,-100\right) \text{ and } ub^\top=\left(100,100\right).
	\end{align*}
		\item[{\normalfont Problem 16 (I-MOP7).}\label{I-MOP7}] Here, $m=3,n=2$, and 
	\begin{align*}
		&G_1(x_1,x_2):=\left[\tfrac{1}{4},\tfrac{1}{2}\right]\odot \left(x_1-2\right)^2\oplus \left[\tfrac{1}{26},\tfrac{1}{13}\right]\odot \left(x_2+1\right)^2\oplus\left[2,3\right] ,\\
		&G_2(x_1,x_2):=\left[\tfrac{1}{9},\tfrac{1}{4}\right]\odot \left(x_1+x_2-3\right)^2\oplus \left[\tfrac{1}{16},\tfrac{1}{8}\right]\odot \left(-x_1+x_2+2\right)^2\ominus_{gH}\left[17,20\right],\\
		&G_3(x_1,x_2):=\left[\tfrac{1}{25},\tfrac{1}{7}\right]\odot \left(x_1+2x_2-1\right)^2\oplus \left[\tfrac{1}{34},\tfrac{1}{17}\right]\odot \left(-x_1+2x_2\right)^2\ominus_{gH}\left[13,15\right], \\
		&lb^\top=\left(-400,-400\right) \text{ and } ub^\top=\left(400,400\right).
	\end{align*}
		\item[{\normalfont Problem 17 (I-VFM2).}\label{I-VFM2}] Here, $m=3,n=3$, and 
	\begin{align*}
			&G_1(x_1,x_2,x_3):=\left[0.1,0.2\right]\odot x_1^2\oplus \left[0.1,0.3\right]\odot x_2^2\oplus \left[0.1,0.2\right]\odot x_3^2,\\
			&G_2(x_1,x_2,x_3):=\left[0.1,0.3\right]\odot \left(x_1-5\right)^2\oplus \left[0.1,0.5\right]\odot \left(x_2-5\right)^2\oplus \left[0.1,0.4\right]\odot \left(x_3-5\right)^2,\\
			&G_3(x_1,x_2,x_3):=\left[0.1,0.2\right]\odot x_1^2\ominus_{gH} \left[0.1,0.3\right]\odot x_2^2 \oplus \left[0.1,0.2\right]\odot x_3^2, \\
			&lb^\top=\left(-5,-5,-5\right) \text{ and } ub^\top=\left(10,10,10\right).
	\end{align*}
	\item[{\normalfont Problem 18 (I-TR1).}\label{I-TR1}] Here, $m=3,n=3$, and 
	\begin{align*}
		&G_1(x_1,x_2,x_3):=\left[15,30\right]\ominus_{gH} \left[0.1,0.3\right]\odot \left(x_1^3+x_1^2\left(1+x_2+x_3\right)+x_2^3+x_3^3\right),\\
		&G_2(x_1,x_2,x_3):=\left[25,45\right]\ominus_{gH} \left[0.1,0.2\right]\odot \left(x_1^3+2x_2^3+x_2^2\left(2+x_1+x_3\right)+x_3^3\right),\\
		&G_3(x_1,x_2,x_3):=\left[30,60\right]\ominus_{gH} \left[0.1,0.3\right]\odot \left(x_1^3+x_2^3+3x_3^3+x_3^2\left(3+x_1+x_2\right)\right), \\
		&lb^\top=\left(1,1,1\right) \text{ and } ub^\top=\left(4,4,4\right).
	\end{align*}
	\item[{\normalfont Problem 19 (I-AP4).}\label{I-AP4}] Here, $m=3,n=3$, and 
	\begin{align*}
		&G_1(x_1,x_2,x_3):=\left[\tfrac{1}{9},\tfrac{1}{3}\right]\odot \left(\left(x_1-1\right)^4+2\left(x_2-2\right)^4+3\left(x_3-3\right)^4\right),\\
		&G_2(x_1,x_2,x_3):=\left[2,3\right]\odot\exp\left(\tfrac{x_1+x_2+x_3}{3}\right)\oplus\left[2,5\right]\odot\left(x_1^2+x_2^2+x_3^2\right),\\
		&G_3(x_1,x_2,x_3):=\left[\tfrac{1}{4},\tfrac{1}{3}\right]\odot\left(3\exp\left(-x_1\right)+4\exp\left(-x_2\right)+3\exp\left(-x_3\right)\right), \\
		&lb^\top=\left(-100,-100,-100\right) \text{ and } ub^\top=\left(100,100,100\right).
	\end{align*}
		\item[{\normalfont Problem 20 (I-Comet).}\label{I-Comet}] Here, $m=3,n=3$, and 
	\begin{align*}
		&G_1(x_1,x_2,x_3):=\left[1,1.5\right]\odot\left(1+x_3\right)\odot\left(x_1^3x_2^2-10x_1-4x_2\right),\\
		&G_2(x_1,x_2,x_3):=\left[1,1.5\right]\odot\left(1+x_3\right)\odot\left(x_1^3x_2^2-10x_1+4x_2\right),\\
		&G_3(x_1,x_2,x_3):=\left[0.2,1\right]\odot\left(1+x_3\right)\odot x_1^2, \\
		&lb^\top=\left(1,-2,0\right) \text{ and } ub^\top=\left(3.5,2,1\right).
	\end{align*}
\end{description}

\end{appendices}

\noindent {\bf Acknowledgments}  
T. Mondal and D. Ghosh  acknowledge the Core Research Grant (CRG/2022/001347) from Science and Engineering Research Board, India, to carry out this research work. D. S. Kim is thankful to the National Research Foundation of Korea (NRF) grant funded by the Korean government (MSIT) (RS-2025-19622979).  \\

\noindent {\bf Data Availability}\\ 
No data was used for the research described in the article. \\ 

\noindent {\bf Disclosure Statement} \\ 
The authors do not have any conflicts of interest.

\end{document}